\documentclass[12pt]{article}
\usepackage{latexsym}
\usepackage{amssymb, amsmath, xspace, lscape,  latexsym}
\usepackage{amsthm}

\newcommand{\bea}{\begin{eqnarray}}
\newcommand{\eea}{\end{eqnarray}}
\def\beq#1#2\eeq{
        \begin{equation}
        \label{#1}
            #2
        \end{equation}}

\newcommand{\bt}{\beta}

\renewcommand{\hat}{\widehat}

\def\btheor#1\etheor{
        \begin{theor}
            #1
        \end{theor}
    }

    \def\bsled#1\esled{
        \begin{sled}
            #1
        \end{sled}   }

\newtheorem{theorem}{Theorem}

\newtheorem{lemma}{Lemma}

\newtheorem{cor}{Corollary}

\def\hm#1{#1\nobreak\discretionary{}{\hbox{\m@th$#1$}}{}}
\def\mi#1{\discretionary{\hbox{\m@th$#1$}}{\hbox{\m@th$#1$}}{}}

\vspace{4ex}

\begin{document}
\title{\bf Perturbed Hankel determinant, correlation functions and Painlev\'{e} equations}
\author{Min Chen$^{1}$\thanks{chenminfst@gmail.com}, Yang Chen$^{1}$\thanks{Corresponding author(Yang Chen), yayangchen@umac.mo and yangbrookchen@yahoo.co.uk}  and  Engui Fan$^{2}$\thanks{faneg@fudan.edu.cn}\\
        {\normalsize $^{1}$Department of Mathematics, University of Macau, Avenida da Universidade,}\\
        {\normalsize Taipa, Macau, P.R. China}\\
        {\normalsize $^{2}$School of Mathematical Science, Fudan University, Shanghai 200433, P.R. China}\\
        }
\date{}

\maketitle
\begin{abstract}
We continue with the study of the Hankel determinant,
$$
D_{n}(t,\alpha,\beta):=\det\left(\int_{0}^{1}x^{j+k}w(x;t,\alpha,\beta)dx\right)_{j,k=0 }^{n-1},
$$
generated by a Pollaczek-Jacobi type weight,
$$
w(x;t,\alpha,\beta):=x^{\alpha}(1-x)^{\beta}{\rm e}^{-t/x},
\quad x\in [0,1], \quad \alpha>0, \quad \beta>0, \quad t\geq 0.
$$
This reduces to the ``pure"  Jacobi weight at $t=0.$ We may take $\alpha\in \mathbb{R}$, in the situation while $t$ is strictly greater than $0.$ It was shown in Chen and Dai (2010), that the logarithmic
derivative of this Hankel determinant satisfies a Jimbo-Miwa-Okamoto $\sigma$-form of Painlev\'e \uppercase\expandafter{\romannumeral5} (${\rm P_{\uppercase\expandafter{\romannumeral5}}}$). In fact the logarithmic of the Hankel determinant has an integral representation in terms of a particular ${\rm P_{\uppercase\expandafter{\romannumeral5}}}.$
\\
In this paper, we show that, under a double scaling, where $n$ the dimension of the Hankel matrix tends to $\infty$, and $t$ tends to $0^{+},$ such that $s:=2n^2t$ is finite, the double scaled Hankel determinant (effectively an operator determinant) has an integral representation in terms of a particular ${\rm P_{\uppercase\expandafter{\romannumeral3}'}}.$ Expansions of the scaled Hankel determinant for small and large $s$ are found. A further double scaling with $\alpha=-2n+\lambda,$ where $n\rightarrow \infty$ and $t,$ tends to $0^{+},$ such that
$s:=nt$ is finite. In this situation the scaled Hankel determinant has an integral representation in terms of a particular ${\rm P_{\uppercase\expandafter{\romannumeral5}}},$
and its small and large $s$ asymptotic expansions are also found.
\\
The reproducing kernel in terms of monic polynomials orthogonal with respect to the Pollaczek-Jacobi type weight, under the origin (or hard edge) scaling may be expressed
in terms of the solutions of a second order linear ordinary differential equation (ODE). With special choices of the parameters, the limiting (double scaled) kernel and the second order ODE degenerate to Bessel kernel and the Bessel differential equation, respectively.
\\
We also applied this method to polynomials orthogonal with respect to the perturbed Laguerre weight;
$
w(x;t,\alpha):=x^{\alpha}{\rm e}^{-x}\:{\rm e}^{-t/x},
$
$
0\leq x<\infty,\;\;\;\alpha>0,\;\;t>0.
$
The scaled kernel at origin of this perturbed Laguerre ensemble has the same behavior with the above limiting kernel, although difference scaled schemes are adopted on these two kernels.
\end{abstract}
\vfill\eject

\setcounter{equation}{0}
\section{Introduction}
The determinant of the $n\times n$ Hankel matrix,
$$
\left(\int_{\mathbb{L}}w(x)x^{j+k}dx\right)_{0\leq j,k\leq n-1},
$$
has an equivalent representation as the multiple integral \cite{M2004},
\begin{equation}\label{00a2}
D_{n}[w]=\frac{1}{n!}\int_{\mathbb{L}^{n}}\prod_{1\leq j<k\leq n}\left(x_{j}-x_{k}\right)^{2}\prod_{\ell=1}^{n}w(x_{\ell})dx_{\ell}, \quad {\rm with} \quad w(x_{\ell})=e^{-{\rm v}(x_{\ell})},
\end{equation}
where $w(x)$ is a positive weight function supported on $\mathbb{L}$ ($\subset\mathbb{R}$) and  ${\rm v}(x)$ is known as the external potential. Hankel determinant a fundamental object unitary random matrix theory \cite{M2004} with many applications, in mathematics, physics and other areas, for example, in wireless communications \cite{BCL2009}.
For instance, such determinant that arises from the singularly perturbed Laguerre weight, becomes the moment generating function of certain linear statistics \cite{ChenIts12010}. Such
determinant also appears in the computation of the Wigner delay time distribution in chaotic cavities, studied from the point of view of
large derivations\cite{TM2013}. We refer the Reader to \cite{B1997, BCW2001, CC2014} for related material.

Here are some well-known facts on orthogonal polynomials.
\\
The joint probability density function of the eigenvalues $x_{1}, x_{2}, \ldots, x_{n}$ of a $n\times{n}$ Hermitian matrix ensemble can found, for example, in \cite{M2004},
$$
p(x_{1},x_{2},\ldots,x_{n})=\frac{1}{D_{n}[w]}\frac{1}{n!}\prod_{1\leq j<k\leq n}(x_{j}-x_{k})^{2}\prod_{\ell=1}^{n}w(x_{\ell}),
$$
\\
From which, he $m$-point correlation function, follows, see \cite{D1970,M2004};
$$
R_{m}(x_{1},x_{2},\ldots,x_{m})=\frac{n!}{(n-m)!}\int_{\mathbb{L}^{n-m}}p(x_{1},x_{2},\ldots,x_{n})dx_{m+1}\cdots{dx_{n}}.
$$
An equivalent expression reads,
$$
R_{m}(x_{1},x_{2},\ldots,x_{m})=\det\left(K_{n}(x_{k},x_{j})\right)_{1\leq{k},j\leq{m}}.
$$
Here kernel $K_{n}(x,y)$ is defined by the monic polynomials $P_{n}(x)$ orthogonal with respect to the weight $w(x)$ on $\mathbb{L}$,
\begin{equation}\label{451}
K_{n}(x,y):=\sqrt{w(x)}\sqrt{w(y)}\sum_{j=0}^{n-1}\frac{P_{j}(x)P_{j}(y)}{h_{j}},
\end{equation}
\begin{equation}\label{452}
\int_{\mathbb{L}}P_{n}(x)P_{m}(x)w(x)dx=h_{n}\delta_{nm},
\end{equation}
and $h_{n}$ is the square of the $L^{2}$ norm. An immediate consequence of the orthogonality relations is the reproducing property;
$$
K_{n}(x,y)=\int_{\mathbb{L}}K_n(x,z)K_n(z,y)dz.
$$
Further more, the monic orthogonal polynomials satisfy three terms recurrence relations,
$$
xP_{n}(x)=P_{n+1}(x)+\alpha_{n}P_{n}(x)+\beta_{n}P_{n-1}(x),
$$
subjected to the initial data, $P_{0}(x)=1$, and $\beta_{0}P_{-1}(x)=0$.
\\
With the aid of the Christoffel--Darboux formula \cite{Szego1939}, which is an immediate consequence of the three term recurrence relations, the kernel has a simple closed form
\begin{equation}\label{00a1}
K_{n}(x,y)=\sqrt{w(x)}\sqrt{w(y)}\frac{P_{n}(x)P_{n-1}(y)-P_{n}(y)P_{n-1}(x)}{h_{n-1}(x-y)}.
\end{equation}
It is of interest to investigate the feature of local eigenvalue correlation through a description of the correlation kernel, for $n$. For example,
in the case of the Gaussian Unitary Ensemble (GUE), where $w(x)={\rm e}^{-x^2},\;\;x\in \mathbb{R}$ the limiting mean eigenvalue density reads,
$$
\lim_{n\rightarrow\infty}\sqrt{\frac{2}{n}}R_{1}\left(\sqrt{2n}x\right)=\frac{2\sqrt{1-x^{2}}}{\pi},\quad -1<x<1.
$$
This is the Wigner semi-circle law. Re-scaling with respect to a fix point $x_{0},$  leads to the sine kernel
$$
\lim_{n\rightarrow\infty}\frac{\pi}{\sqrt{2n}}K_{n}\left(x_{0}+\frac{\pi{x}}{\sqrt{2n}},x_{0}+\frac{\pi{y}}{\sqrt{2n}}\right)=\frac{\sin\pi(x-y)}{\pi(x-y)}.
$$
The limiting kernel becomes the Airy kernel with a suitable re-scaling at the edge $\sqrt{2n},$ the edge of the eigenvalue spectrum, obtained by Tracy and Widom \cite{TW1993AK},
$$
K_{\rm{Airy}}(x,y)=\lim_{n\rightarrow\infty}\frac{1}{2^{\frac{1}{2}}n^{\frac{1}{6}}}K_{n}\left(\sqrt{2n}+\frac{x}{2^{\frac{1}{2}}n^{\frac{1}{6}}},
\sqrt{2n}+\frac{y}{2^{\frac{1}{2}}n^{\frac{1}{6}}}\right)=\frac{Ai(x)Ai'(y)-Ai'(x)Ai(y)}{x-y}.
$$
Here $Ai(z)$ is the Airy function. Tracy and Widom investigated the logarithmic derivatives of operator determinants, involving the Airy kernel,
in the study of the level spacing distribution \cite{TW1993AK}. Chen and Ismail \cite{YMI1997}, obtained the Airy kernel by studying the limiting behavior of kernels
generated in the situations, where $\textsf{v}(x)$ is any even degree polynomials in $x$ with positive coefficient in the highest order monomial.
\\
It is a useful technique to characterize the large $n$ behavior of the scaled kernel in terms of differential equation. Tracy and widom \cite{TW1994BK} adopted tools from integrable system to analyze the correlation kernel in the large $n$ limit, and scale at hard edge of the Laguerre unitary ensemble (LUE). The limiting kernel can be expressed by the regular solutions of the Bessel differential equation. We restate the limiting kernel here, 
 \cite{TW1994BK}, 
\begin{equation}\label{b0b1}
K_{\rm Bessel}(x,y)=\frac{\phi(x)y\phi'(y)-x\phi'(x)\phi(y)}{x-y},
\end{equation}
where $\phi(x)$ is the regular solutions of the Bessel differential equation ((2.14), \cite{TW1994BK}),
\begin{equation}\label{b0b2}
x^{2}\phi''(x)+x\phi'(x)+\frac{1}{4}(x-\alpha^{2})\phi(x)=0,
\end{equation}
namely $\phi(x)=\sqrt{\mu}J_{\alpha}(\sqrt{x}),$ $\mu$ is a parameter, and $J_{\alpha}(z)$ is the Bessel function with order $\alpha,$
and
\begin{align}\label{00a}
K_{\rm Bessel}\left(x,y\right)&=\mu\frac{J_{\alpha}(\sqrt{x})\sqrt{y}J'_{\alpha}(\sqrt{y})-\sqrt{x}J'_{\alpha}(\sqrt{x})J_{\alpha}(\sqrt{y})}{2(x-y)},\\
K_{\rm Bessel}(x,x)&=\frac{\mu}{4}\left(J_{\alpha}(\sqrt{x})^{2}-J_{\alpha+1}(\sqrt{x})J_{\alpha-1}(\sqrt{x})\right).
\end{align}
This is known as the Bessel kernel. See \cite{TW1994FD} for further information.
\\
From the formulas of Laguerre and Hermite polynomial, Forrester \cite{F1993} obtains the Bessel kernel and the Airy kernel after suitable re-scaling. Moreover, Nagao and Wadati \cite{TW1991} deduced the Bessel kernel by scaling the Jacobi ensemble at the hard edges, $\pm1$. Kuijlaars and Zhang \cite{KZ2014} obtain a limiting kernel as a generalization of Bessel kernel by scaling the correlation kernel of complex Ginibre random matrices at the hard edge, see the references therein for more information.
\\
The theory of integrable kernels was put forward in \cite{IIKS1990}. A condition for a kernel $K(x,y)$ to be integrable, is that it can be expressed
 as the sum of functions  $f_{k}(x)$ and $g_{k}(x)$, that is,
$$
K(x,y)=\frac{1}{x-y}\sum_{k=0}^{p}f_{k}(x)g_{k}(y),{\rm \;\;\; where\;\;} \quad\sum_{k=0}^{p}f_{k}(x)g_{k}(x)=0.
$$
The Sine, Airy and Bessel kernel are all integrable.
\\
In our approach, we study the Hankel determinant directly, without expressing it in the form of $\det(I_n-K_n).$
\\
For the problem at hand, the Hankel determinant reads,
$$
D_{n}(t,\alpha,\beta):={\rm det}\left(\int_{0}^{1}x^{j+k}w(x;t,\alpha,\beta)dx\right)_{j,k=0}^{n-1},
$$
where
\begin{equation}\label{a0a1}
w(x;t,\alpha,\beta)=x^{\alpha}(1-x)^{\beta}e^{-t/x},\quad x\in[0,1],\quad t \geq 0, \quad \beta >0, \quad \alpha>0,
\end{equation}
is the Pollaczek-Jocobi type weight.

For $t>0,$  ${\rm e}^{-t/x}\rightarrow 0,$  as $x\rightarrow  0,$  with far greater speed than $x^{\alpha}$ tends to 0, if $\alpha>0.$
The same can be said, for $\alpha<0,$ as long as $t>0.$
The Pollaczek-Jacobi type  weight violates the Szeg\"{o} condition (see \cite{Szego1939} and \cite{CD2010}),
which reads,
$$
\int_{0}^{1}\frac{|\ln w(x)|}{\sqrt{x(1-x)}}dx<\infty.
$$
\\
Any monic  polynomial orthogonal with respect to some weight can be represented by Heine's multiple-integral, and in our case,
$$
P_{n}(z;t,\alpha,\beta)=\frac{1}{n!D_{n}(t,\alpha,\beta)}\int_{(0,1)^{n}}\prod_{m=1}^{n}(z-x_{m})\prod_{1\leq j<k\leq n}(x_{j}-x_{k})^{2}\prod_{\ell=1}^{n}w(x_{\ell};t,\alpha,\beta)dx_{\ell}.
$$
The constant term of our orthogonal polynomial has the closed form expression,
$$
(-1)^{n}P_{n}(0;t,\alpha,\beta)=\frac{D_{n}(t,\alpha+1,\beta)}{D_{n}(t,\alpha,\beta)}.
$$
The remainder of this paper is organized as follow. In Section 2.1, our interest lies in a double scaling analysis, where $t\rightarrow 0^{+},$ and $n\rightarrow \infty$, such that $s=2n^{2}t$ is finite. we shall see later that that double-scaled and in some sense infinite dimensional Hankel determinant has an integral representation in terms of a particular
${\rm P_{{III}^{'}}}.$ Its logarithmic derivative satisfies a particular Jimbo-Miwa-Okamoto $\sigma$-form Painlev\'{e}. This double scaling analysis is based on results of Chen and Dai \cite{CD2010}, obtained in the finite $n$ situation. In Section 2.2, we obtained the asymptotic expansions of scaled  Hankel determinant as formal series for small and large $s$. Furthermore, the constant in the large $s$ expansion is found later, in Section 2.4. In Section 2.3, we introduce a new double scaling scheme, where
 $t\rightarrow 0^{+},$ $n\rightarrow \infty,$ $\beta=\widetilde{\beta}n\rightarrow \infty,$ with ${\it fixed}\;\widetilde{\beta}$ and $s=2(1+\widetilde{\beta})n^{2}t$ is finite. In this
  case the infinite dimensional determinant may be  characterized by the same Painlev\'{e} equations that appeared theorem 3, after a minor change of variable. In Section 2.4,
 an evaluation is made on the constant term of the monic orthogonal polynomials$P_{n}(0;t,\alpha,\beta)$, for large $n$, and $s=2n^{2}t$. This comes from an application of the Szeg\"o limit theorem for Toeplitz determinants, but adapted to Hankel determinants.
 From these results, the constant $c(\alpha)$ appears in the asymptotic expansion of the double-scaled Hankel determinant, for large $s$ is found. In Section 2.5, we combine Normand's formulas \cite{JMN2004} with method of Forrester and Witte \cite{PFW2006}, to find asymptotic expansions of the Hankel determinant for small $s,$ in agreement with
 our expansion in Section 2.2.
\\
In Section 3, we propose another double scaling scenario, namely, $\alpha=-2n+\lambda,$
\\
 $t\rightarrow0^{+},$ $n\rightarrow \infty,$ where $\lambda<0.$  In such a way that
 $s=nt>0$ is finite. We note that our Hankel determinant under this double scaling scheme can be reproduced by another perturbed Laguerre weight, namely,
  $$
  w(x;t)=(x+t)^{-(\lambda+\beta)}x^{\beta}\:{\rm e}^{-x-t},\;\;\lambda<0,\;\beta>0,\;t \geq 0,\;\;\;x\in(0,\infty),
  $$
  studied in \cite{YM2012}.
  In this case, the Hankel determinant has an integral representation in terms of a particular ${\rm P_{V}}$, (equivalent to a ${\rm P_{III}}$). From which we determine its  small and large $s$ expansion. Moreover, the constant which appear in the asymptotic expansion of the double-scaled Hankel determinant, for large s, is found.
\\
In Section 4, combining the ladder operator relations in $x$ obtained in \cite{CD2010} and further ladder operator relations in $t$ obtained here,
 satisfied by the monic orthogonal polynomials $P_{n}(x;t,\alpha,\beta)$
(see theorem 12),
we construct a Lax pair, involving derivative in $x$ and derivative in $t$.  The natural compatibility condition, reproduces certain results of \cite{CD2010}. In order to analyse
the limiting behavior of the kernel arising form the Pollaczek-Jacobi type weight, a scaling scheme is introduced. Here, $t\rightarrow 0^{+},$ $n\to\infty,$ such that
$s=2n^2t$ is finite. The ``coordinates", $x$ and $y$ have been re-scaled to $x=\frac{\zeta}{4n^2},$ $y=\frac{\zeta*}{4n^2}$. Ultimately, this shows that the limiting
kernel may be characterized by solutions a
second order ODE. If $s=0$, this essentially reduces to the Jacobi weight, and the limit kernel and the second order ODE reduce to the Bessel kernel and Bessel differential equation, respectively.
\\
In Section 5, we adopt the method in Section 4 to study the kernel arising form  the singularly perturbed Laguerre weight. We adopt another scaling scheme, where
$t\to 0^+$, $n\to\infty,$ such that $s=(2n+1+\alpha)t$ is finite. The ``coordinates", $x$ and $y$  have been re-scaled to $x=\frac{\zeta}{4n}$ and $y=\frac{\zeta^*}{4n}.$

\section{Double scaling analysis \uppercase\expandafter{\romannumeral1}.}
Chen and Dai \cite{CD2010} applied the ladder operator method to investigate the Hankel determinant obtained from the Pollacaek-Jacobi type weight. It was found that the logarithmic derivative of the Hankel determinant satisfies a particular Jimbo-Miwa-Okamoto $\sigma$-form of Painlev\'{e} equation.
\\
An immediate consequence of the relationships obtained in \cite{CD2010}, is that the  Hankel determinant has an
 integral representation in terms of a ${\rm P_{V}}$ transcendent in the variable $y(t,\alpha,\beta)$. See Lemma 1.
  We shall be concerned with the behavior of the
 Hankel determinant, as $n$, the dimension of the Hankel matrix tends to infinity. For this purpose, a double scaling scheme is introduced, namely, sending $n\rightarrow\infty,$ $t\rightarrow0^{+},$ and such that $s:=2n^{2}t$ remain fixed.  We recall  theorem 5.4 in \cite{CD2010}.
\begin{theorem}
The logarithmic derivative of the Hankel determinant with respect to $t$,
\bea\label{b42}
H_{n}(t,\alpha,\beta):=t\frac{d}{dt}\ln{\frac{D_{n}(t,\alpha,\beta)}{D_{n}(0,\alpha,\beta)}}=(2n+\alpha+\beta)(r_n^{*}(t)-r_n(t)),
\eea
satisfies the following ordinary differential equation:
\bea\label{b2}
(tH_{n}^{''})^{2}=[n(n+\alpha+\beta)-H_{n}+(\alpha+t)H_{n}^{'}]^{2}+4H_{n}^{'}(tH_{n}^{'}-H_{n})(\beta-H_{n}^{'}),
\eea
with the initial data $H_{n}(0,\alpha,\beta)=0.$ 
\end{theorem}
Here $r_n^{*}(t)$ and $r_n(t)$ variables defined in \cite{CD2010}.
\\
By a change of variable, the above ODE turns out to be a particular Jimbo-Miwa-Okamoto $\sigma$-from of ${\rm P_{V}}$. See \cite{CD2010} for more details.
\\
We now recall the Theorem 7.2 in \cite{CD2010} and replace $S_n(t)$ in that paper by   $y(t,\alpha,\beta).$
\begin{theorem}
Let
\bea\label{b8}
y(t,\alpha,\beta):=\frac{R_{n}(t)}{2n+1+\alpha+\beta}.
\eea
Then $y(t,\alpha,\beta)$ satisfies the following differential equation:
\begin{align}\label{b3}
y''=&\frac{3y-1}{2y(y-1)}(y')^2-\frac{y'}{t}+\frac{(2n+1+\alpha+\beta)^2(y-1)^2y}{2t^2}-\frac{(y-1)^2\beta^2}{2t^2y}
+\frac{\alpha{y}}{t}-\frac{y(y+1)}{2(y-1)},
\end{align}
which is a ${\rm P_{V}\left((2n+1+\alpha+\beta)^{2}/2,-\beta^{2}/2,\alpha,-1/2\right)}.$ The boundary condition is $y(0,\alpha,\beta)=1.$
\end{theorem}
\par
From results obtained in \cite{CD2010}, we show that the Hankel determinant has an integral representation in terms of ${\rm P_{V}}.$ See the Lemma below.
\begin{lemma}
The logarithmic derivative of the Hankel determinant $H_{n}(t,\alpha,\beta),$ and  $r_{n}^{*}$, can be expressed in terms of $y(t,\alpha,\beta)$ and $y'(t,\alpha,\beta)$ as follow,
\begin{align}\label{b4}
H_{n}(t,\alpha,\beta)&=t\frac{d}{dt}\ln{\frac{D_{n}(t,\alpha,\beta)}{D_{n}(0,\alpha,\beta)}}\nonumber\\
&=-\frac{1}{4y(y-1)^{2}}\left[\beta^{2}+2y^{3}-2(2n+\alpha+\beta)^{2}y^{3}-2(t+\beta)(2n+\alpha+\beta)y^{3}-y^{4}-t^{2}{y'}^{2}\right.\nonumber\\
&\left.+(2n+\alpha+\beta)^{2}y^{4}-4n(t+\beta)y-2\beta(t+\alpha+2\beta)y-2ty'y+2ty'y^{2}+4n^{2}y^{2}\right.\nonumber\\
&\left.+y^{2}\left((t+\alpha)^{2}+4t\beta+6\alpha{\beta}+6\beta^{2}-1\right)+4n(2t+\alpha+3\beta)y^{2}\right]+n(n+\alpha-t),
\end{align}
with the initial data $H_{n}(0,\alpha,\beta)=0.$
\end{lemma}
\begin{proof}
Recalling identities $(5.8)$, $(5.9)$ and $(7.2)$ in \cite{CD2010} as,
\bea\label{b5}
R_{n}(t)=\frac{(2n+1+\alpha+\beta)[2r_{n}^{2}+(t+2\beta-2r_{n}^{*})r_{n}+(2n+\alpha)r_{n}^{*}-nt-tr_{n}']}{2[(r_{n}^{*}-r_{n})^{2}+(2n+\alpha-t)r_{n}^{*}+(\beta+t)r_{n}-nt]},\nonumber
\eea
\bea\label{b6}
\frac{1}{R_{n}(t)}=\frac{2r_{n}^{2}+(t+2\beta-2r_{n}^{*})r_{n}+(2n+\alpha)r_{n}^{*}-nt+tr_{n}'}{2(2n+1+\alpha+\beta)(\beta+r_{n})r_{n}},\nonumber
\eea
and
\bea\label{b7}
r_{n}^{*}=\frac{1}{2R_{n}}\left[tR_{n}'-(2n+1+\alpha+\beta)(2r_{n}-R_{n}+\beta)\right]+r_{n}-\frac{R_{n}-\beta-t}{2}.\nonumber
\eea
With the aid of above equations and the definition of $y(t,\alpha,\beta)$ in (\ref{b8}), the variables
$r_{n}(t)$ and $r_{n}^{*}(t)$ may be expressed in terms of $y$ and $y'.$ Straightforward computation produces (\ref{b4}).
\end{proof}
Hence the Hankel determinant has an integral representation in terms of $y$ and $y'$, and that $y(t,\alpha,\beta)$ satisfies the ${\rm P_{V}}$ given by (\ref{b3}).
\par

\subsection{Scaling limit of the Hankel determinant in terms of Painlev\'{e} equations.}
Carrying out the double scaling and combining with Theorem 1, Theorem 2 and Lemma 1, we find that the (effectively) infinite dimensional Hankel determinant has an
integral representation in terms of $(\ref{a57}).$ Moreover, the logarithmic derivatives of such a  Hankel determinant satisfies another
$\sigma$-form of the corresponding Painlev\'{e} equation.
\begin{theorem}
Let
\bea\label{a50}
y(t,\alpha,\beta):=1+\frac{f(t,\alpha,\beta)}{n^{2}},\qquad {\rm and} \qquad s:=2n^{2}t.
\eea
$t \rightarrow 0^{+}$ and $n \rightarrow \infty$, such that $s\in(0, \infty).$
\\
If
$$
g(s,\alpha,\beta):=\lim_{n \rightarrow \infty}f\left(\frac{s}{2n^{2}},\alpha,\beta\right),
$$
then $g(s,\alpha,\beta)$ satisfies
\bea\label{a57}
g''=\frac{g'^{2}}{g}-\frac{g'}{s}+\frac{2g^{2}}{s^2}+\frac{\alpha}{2s}-\frac{1}{4g},
\eea
with the initial data $g(0,\alpha,\beta)=0,$ $g'(0,\alpha,\beta)=\frac{1}{2\alpha},$
The equation (\ref{a57}) is ${\rm P_{III'}}(8, 2\alpha, 0, -1)$.
\\
If
$$
\quad {\cal H}(s,\alpha,\beta):=\lim_{n \rightarrow \infty}H_{n}\left(\frac{s}{2n^{2}},\alpha,\beta\right), \quad
$$
then
${\cal H}(s,\alpha,\beta)$ satisfies,
\bea\label{a92}
\left(s{\cal H}''\right)^{2}+4\left({\cal H}'\right)^{2}\left(s{\cal H}'-{\cal H}\right)-\left(\alpha{{\cal H}'}+\frac{1}{2}\right)^{2}=0,
\eea
with the initial conditions ${\cal H}(0,\alpha,\beta)=0$, ${\cal H}'(0,\alpha,\beta)=-\frac{1}{2\alpha}$.

Furthermore, if
$$
\Delta(s,\alpha,\beta):=\lim_{n \rightarrow \infty}\frac{D_{n}\left(s/2n^{2},\alpha,\beta\right)}{D_{n}\left(0,\alpha,\beta\right)},\nonumber
$$
then
\bea\label{b9}
{\cal H}=s\frac{d}{ds}\ln{\Delta(s,\alpha,\beta)}=\frac{\left(s\:g'-g\right)^{2}}{4g^{2}}+\frac{4{\alpha}sg-s^{2}}{16g^{2}}-g-\frac{\alpha^{2}}{4}.
\eea
\end{theorem}
\begin{proof}
Substituting the definition (\ref{a50}) into $(\ref{b3})$, then we see that $g(s,\alpha,\beta)$
satisfied  ${\rm P_{III'}}(8, 2\alpha, 0, -1)$. See $(\ref{a57}).$
Plugging $s=2n^{2}t$ into $(\ref{b2}),$ we see that the limit of $H_{n}(t,\alpha,\beta)$ satisfies a $\sigma$-from Painlev\'{e} equation $(\ref{a92}).$
Moreover, substituting $(\ref{a50})$ into $(\ref{b4}),$ then the equation $(\ref{b9})$ is found.
\end{proof}
\par
Note that the differential equation is satisfied by $g(s,\alpha,\beta)$ is a ${\rm P_{III'}}(8, 2\alpha, 0, -1)$, see \cite{GLS2002, KSO2006}. From ${\rm P_{III'}}(8, 2\alpha, 0, -1)$ and the $\sigma$-form of the  Painlev\'{e} equation, together with the boundary conditions, $g(s,\alpha,\beta)$ and ${\cal H}(s,\alpha,\beta)$ are {\it independent} of $\beta$. We use these notations to distinguish this result from those previously obtained \cite{CC2014}. The $\sigma$-from Painlev\'{e} equation of $(\ref{a92})$ has the same  double scaling limit of the logarithmic derivative of the Hankel determinant generated by the singularly perturbed Laguerre weight,
$x^{\alpha}{\rm e}^{-x-t/x}$, where $x\geq 0,\;t>0,$ and real $\alpha,$ however, the double scaling scheme is different from that in \cite{CC2014}.

{\bf Remark 1:}
By the change variables
\bea\label{a60}
F(x,\alpha,\beta)=\frac{8}{x}g\left(\frac{x^{2}}{8},\alpha,\beta\right), \nonumber
\eea
it is seen that the ${\rm P_{III'}}(8, 2\alpha, 0, -1)$ satsified by $g$ becomes,
\bea\label{a61}
F''=\frac{F'^{2}}{F}-\frac{F'}{x}+\frac{F^{2}}{x}+\frac{2\alpha}{x}-\frac{1}{F},\nonumber
\eea
  a ${\rm P_{III}}(1, 2\alpha, 0, -1)$. See \cite{GLS2002}. The $C$ potential (2.22) introduced in
\cite{CC2014}, after a change of variable, satisfies the above equation.
\par
There are three algebraic solutions of ${\rm P_{III'}}(8, 2\alpha, 0, -1),$
\bea\label{a66}
g(s,\alpha,\beta)=\frac{1}{2}s^{\frac{2}{3}},\quad {\rm for}\quad \alpha=0. \nonumber
\eea
\bea\label{a67}
g(s,\alpha,\beta)=\frac{1}{2}s^{\frac{2}{3}}\mp\frac{1}{6}s^{\frac{1}{3}},\quad {\rm for}\quad \alpha=\pm1. \nonumber
\eea

\subsection{Asymptotic expansions of the scaled Hankel determinant.}
We assume that the solution of ${\rm P_{III'}}(8, 2\alpha, 0, -1)$ for $s\rightarrow0^{+}$ has the power series expansion $\sum_{j=0}^{\infty}a_{j}s^{j},$
with $g(0)=0,$ $g'(0)=\frac{1}{2\alpha},$ and substitute this into (\ref{a57}), by some straightforward computations, one finds,
\begin{align}\label{a72}
g(s,\alpha,\beta)=&\frac{1}{2\alpha}s-\frac{1}{2\alpha^{2}(\alpha^{2}-1)}s^{2}+\frac{3}{2\alpha^{3}(\alpha^{2}-4)(\alpha^{2}-1)}s^{3}
+\frac{9-6\alpha^{2}}{\alpha^{4}(\alpha^{2}-1)^{2}(\alpha^{2}-4)(\alpha^{2}-9)}s^{4}\nonumber\\
&+\frac{5(-36+11\alpha^{2})}{2\alpha^{5}(\alpha^{2}-1)^{2}(\alpha^{2}-4)(\alpha^{2}-9)(\alpha^{2}-16)}s^{5}+\mathcal{O}({s^{6}}), \;\; {\rm where\;\;}\alpha\neq \mathbb{Z}.
\end{align}
\par
For large and positive $s$, we assume that the solution of (\ref{a57}) has the following expansion $ \sum_{j=-2}^{\infty}b_{j}s^{-\frac{j}{3}}.$
The first term of the expansion is $s^{2/3}/2$. A straight forward computation gives,
\begin{align}\label{a74}
g(s,\alpha,\beta)=&\frac{1}{2}s^{\frac{2}{3}}-\frac{\alpha}{6}s^{\frac{1}{3}}+\frac{\alpha(\alpha^{2}-1)}{162}s^{-\frac{1}{3}
}+\frac{\alpha^{2}(\alpha^{2}-1)}{486}s^{-\frac{2}{3}}+\frac{\alpha(\alpha^{2}-1)}{486}s^{-1}\nonumber\\
&-\frac{\alpha^{2}(\alpha^{2}-1)(2\alpha^{2}-11)}{6561}s^{-\frac{4}{3}}+\mathcal{O}({s^{-\frac{5}{3}}}).
\end{align}
\par
Note that this solution becomes the algebraic solutions mentioned in the Remark 1, for $\alpha=0$, and $\alpha=\pm1$
\par
In the next Theorem we obtain asymptotic expressions of the scaled Hankel determinant for small $s$ and large $s$.
\begin{theorem}
Under the double scaling scheme, the asymptotic expansions of the scaled Hankel determinant generated by the Pollaczek-Jacobi type weight, has the following small and large $s$ expansions:
\par
For small $s$,
\begin{align}\label{b27}
\Delta(s,\alpha,\beta)=&\exp\left\{-\frac{s}{2\alpha}+\frac{s^{2}}{8\alpha^{2}(\alpha^{2}-1)}-\frac{s^{3}}{6\alpha^{3}(\alpha^{2}-1)(\alpha^{2}-4)}
+\frac{3(2\alpha^2-3)s^{4}}{16\alpha^{4}(\alpha^{2}-1)^{2}(\alpha^{2}-4)(\alpha^{2}-9)}\right.\nonumber\\
&\left.+\frac{(36-11\alpha^{2})s^{5}}{10\alpha^{5}(\alpha^{2}-1)^{2}(\alpha^{2}-4)(\alpha^{2}-9)(\alpha^{2}-16)}+\mathcal{O}({s^{6})}\right\},
\end{align}
where $\alpha \notin \mathbb{Z}$.
\par
For large $s$,
\begin{align}\label{b30}
\Delta(s,\alpha,\beta)=&\exp\left[c-\frac{9}{8}s^{\frac{2}{3}}+\frac{3\alpha}{2}s^{\frac{1}{3}}+\frac{1-6\alpha^{2}}{36}\ln{s}+\frac{\alpha(1-\alpha^{2})}{18}s^{-\frac{1}{3}}+\frac{\alpha^{2}(1-\alpha^{2})}{216}s^{-\frac{2}{3}}\right.\nonumber\\
&\left.+\frac{\alpha(1-\alpha^{2})}{486}s^{-1}+\mathcal{O}(s^{-\frac{4}{3}})\right],
\end{align}
where $c=c(\alpha)$ is an integration constant, independent of $s$.
\end{theorem}
\begin{proof}
By $(\ref{b9})$, we see that
$$
\ln{\Delta(s,\alpha,\beta)}=\int_{0}^{s}\left(\frac{\left(\xi{g'}-g\right)^{2}}{4\xi{g^{2}}}+\frac{4{\alpha}\xi{g}-\xi^{2}}{16\xi{g^{2}}}-\frac{4g+\alpha^{2}}{4\xi}\right)d\xi.
$$
For small $s$, with $g(s,\alpha,\beta)$  given by $(\ref{a72})$  the asymptotic expansion $\Delta(s,\alpha,\beta)$, $(\ref{b27})$ follows immediately. Similarly, for large $s$, with $g$ given by $(\ref{a74})$, the equation $(\ref{b30})$ is obtained, following straightforward computations.
\end{proof}
Note that the scaling limit of Hankel determinant via from the Pollaczek-Jacobi type weight is independent of $\beta$.
\par
In what follows, we give an account which will ultimately determine $c(\alpha).$
We start from,
\begin{equation*}
(-1)^{n}P_{n}(0; t, \alpha, \beta)=\frac{D_{n}(t,\alpha+1, \beta)}{D_{n}(t, \alpha, \beta)}.
\end{equation*}
\begin{cor}
Sending $n \rightarrow \infty,$ $t \rightarrow 0^{+}$ and such that $s=2n^{2}t,$ is finite, then
\begin{align}\label{a134}
\lim\limits_{n \rightarrow \infty}\frac{(-1)^{n}P_{n}(0; \frac{s}{2n^{2}}, \alpha, \beta)}{(-1)^{n}P_{n}(0; 0, \alpha, \beta)}=&\frac{\Delta(s, \alpha+1, \beta)}{\Delta(s, \alpha, \beta)}=\exp\left(c_{1}+\frac{3}{2}s^{\frac{1}{3}}-\frac{1+2\alpha}{6}\ln{s}-\frac{\alpha(\alpha+1)}{6}s^{-\frac{1}{3}}\right.\nonumber\\
&\left.-\frac{\alpha(\alpha+1)(2\alpha+1)}{108}s^{-\frac{2}{3}}-\frac{\alpha(\alpha+1)}{162}s^{-1}+\mathcal{O}(s^{-\frac{4}{3}})\right),
\end{align}
where $c_{1}=c_{1}(\alpha)$ is a constant, independent of $s,$ and
\begin{equation}\label{aha1}
c_{1}(\alpha)=c(\alpha+1)-c(\alpha)
\end{equation}
and $c(\alpha)$ is the constant in (\ref{b30}).
\end{cor}
\begin{proof}
From the fact
\begin{align*}
\lim\limits_{n \rightarrow \infty}\frac{(-1)^{n}P_{n}(0; \frac{s}{2n^{2}}, \alpha, \beta)}{(-1)^{n}P_{n}(0; 0, \alpha, \beta)}&=\lim\limits_{n \rightarrow \infty}\frac{D_{n}(s/(2n^{2}),\alpha+1, \beta)}{D_{n}(s/(2n^{2}), \alpha, \beta)}\frac{D_{n}(0,\alpha, \beta)}{D_{n}(0,\alpha+1,\beta)}=\frac{\Delta(s,\alpha+1,\beta)}{\Delta(s,\alpha,\beta)},
\end{align*}
the equation (\ref{a134}) and (\ref{aha1}) follow from (\ref{b30}).
\end{proof}
 \par
A computation that produces the constant $c_{1}(\alpha)$ can be found in Section $2.4.$

\subsection{The Hankel determinant for large $\beta$.}
In this subsection, we are interested in the behavior of the Hankel determinant for large $\beta$ and introduce a different scaling process.
Let $n\rightarrow\infty,$ $t \rightarrow 0^{+},$ $\beta:=n\widetilde{\beta},$ $s:=2(1+\widetilde{\beta})n^{2}t,$ such that $s$ and $\widetilde{\beta}$ are fixed,
we then obtain the same Painlev\'{e} equations in the theorem 3; just replace $f(t,\alpha,\beta)$ of $(\ref{a50})$ in the theorem 3 with $f(t,\alpha,\widetilde{\beta})/(1+\widetilde{\beta})$. We state these results in the theorem below.
\par
\begin{theorem}
Let
\bea\label{a58}
\beta:=n\widetilde{\beta},\qquad s:=2(1+\widetilde{\beta})n^{2}t,\qquad {\rm and} \qquad y(t):=1+\frac{f(t,\alpha,\beta)}{(1+\widetilde{\beta})n^{2}},
\eea
$t \rightarrow 0^{+}$ and $n \rightarrow \infty$ such that $s$ and $\widetilde{\beta}$ are finite, $s\in(0,\infty)$ and $\widetilde{\beta}\in(-1,\infty).$
\\
If
$$
g(s,\alpha,\widetilde{\beta}):=\lim_{n \rightarrow \infty}f\left(\frac{s}{2(1+\widetilde{\beta})n^{2}},\alpha,n\widetilde{\beta}\right),
$$
then
$g(s,\alpha,\widetilde{\beta})$ satisfies the following ${\rm P_{III'}(8, 2\alpha, 0, -1)},$
\bea\label{a59}
g''=\frac{g'^{2}}{g}-\frac{g'}{s}+\frac{2g^{2}}{s^2}+\frac{\alpha}{2s}-\frac{1}{4g},
\eea
with initial conditions $g(0,\alpha,\widetilde{\beta})=0$, $g'(0,\alpha,\widetilde{\beta})=\frac{1}{2\alpha}$,
\\
If
$$
\quad {\cal H}(s,\alpha,\widetilde{\beta}):=\lim_{n \rightarrow \infty}H_{n}\left(\frac{s}{2(1+\widetilde{\beta})n^{2}},\alpha,n\widetilde{\beta}\right),
$$
then ${\cal H}(s,\alpha,\widetilde{\beta})$ satisfies,
\bea\label{b32}
\left(s{\cal H}''\right)^{2}+4\left({\cal H}'\right)^{2}\left(s{\cal H}'-{\cal H}\right)-\left(\alpha{{\cal H}'}+\frac{1}{2}\right)^{2}=0,
\eea
with initial conditions ${\cal H}(0,\alpha,\widetilde{\beta})=0$, ${\cal H}'(0,\alpha,\widetilde{\beta})=-\frac{1}{2\alpha}.$
Moreover, if
$$
 \Delta(s,\alpha,\widetilde{\beta}):=\lim_{n \rightarrow \infty}\frac{D_{n}\left(\frac{s}{2(1+\widetilde{\beta})n^{2}},\alpha,n\widetilde{\beta}\right)}{D_{n}\left(0,\alpha,n\widetilde{\beta}\right)},\nonumber
$$
then
\bea\label{bb9}
{\cal H}(s,\alpha,\widetilde{\beta})=s\frac{d}{ds}\ln{\Delta(s,\alpha,\widetilde{\beta})}=\frac{\left(s{g}'-g\right)^{2}}{4g^{2}}+\frac{4{\alpha}s{g}-s^{2}}{16g^{2}}-g-\frac{\alpha^{2}}{4}.
\eea
\end{theorem}
\begin{proof}
Substituting $(\ref{a58})$ into the $(\ref{b3}),$ $g(s,\alpha,\widetilde{\beta})$ is found to satisfy $(\ref{a59}).$  Plugging $\beta=n\widetilde{\beta}$ and $s=2(1+\widetilde{\beta})n^{2}t$ into the $\sigma$-from Painlev\'{e} equation is satisfied by $H_{n}(t,\alpha,\beta)$ in the theorem 1,
 we see that ${\cal H}(s,\alpha,\widetilde{\beta})$ satisfies $(\ref{b32})$. Moreover, substituting $(\ref{a58})$ into $(\ref{b4})$, the equation $(\ref{bb9})$ follows.
\end{proof}
{\bf Remark 2:}
Comparing theorem 3 with theorem 5, one finds that the Painlev\'{e} equations are the same, although we emphasize that their scaling scheme are different from each other.

\subsection{Large $n$ behavior.}
In order to find the constant $c_{1}(\alpha)$ in (\ref{a134}), we need to determine the large $n$ behavior of the constant terms of the monic
orthogonal polynomial, namely, $P_{n}(0;t,\alpha,\beta).$ The Szeg\"o limit theorem  which computes the determinants of the finite section of Toeplitz matrix
with nice symbols, can be adapted to orthogonal polynomials on the line.
\\
The large $n$ computation for the orthogonal polynomials where the potential ${\rm v}$ satisfies the convexity condition \cite{YIsmail1997}, can be found in
\cite{YIsmail1997, YN1998} and also in \cite{JK1998}.
\\
In fact, as $n\rightarrow\infty$,  $P_{n}(z)$ is approximated by
\begin{equation}\label{a112}
P_{n}(z) \sim {\rm exp}\left[-S_{1}(z)-S_{2}(z)\right],
\end{equation}
valid for $z \notin [a, b].$ Here $S_{1}(z)$ and $S_{2}(z)$ are given by ($(4.6)$ and $(4.7)$ in \cite{YN1998}).
These formulas are
\begin{equation}\label{a113}
{\rm exp}(-S_{1}(z))=\frac{1}{2}\left[\left(\frac{z-b}{z-a}\right)^{\frac{1}{4}}+\left(\frac{z-a}{z-b}\right)^{\frac{1}{4}}\right], \quad z \notin [a, b].
\end{equation}
and
\begin{align}\label{a114}
S_{2}(z)=-n\ln&\left(\frac{\sqrt{z-a}+\sqrt{z-b}}{2}\right)^{2}\nonumber \\
&+\frac{1}{2\pi}\int_{a}^{b}\frac{{\rm v}(x)}{\sqrt{(b-x)(x-a)}}\left[\frac{\sqrt{(z-a)(z-b)}}{x-z}+1\right]dx,\quad z \notin [a, b].
\end{align}
For the problem at hand, the Pollaczek-Jacobi type weight of (\ref{a0a1}), then ${\rm v}(x)$ and ${\rm v}'(x)$ are given by,
\bea\label{a4}
{\rm v}(x)=-\ln{w(x)}=\frac{t}{x}-\alpha{\ln{x}}-\beta{\ln(1-x)},\quad {\rm and}\quad {\rm v}'(x)=-\frac{t}{x^{2}}-\frac{\alpha}{x}-\frac{\beta}{x-1}. \nonumber
\eea

Substituting ${\rm v(x)}$ and ${\rm v}'(x)$ into the following identities,
\bea\label{a2a}
\int_{a}^{b}\frac{{\rm v'}(x)}{\sqrt{(b-x)(x-a)}}dx=0,
\eea
and
\bea\label{a3}
\int_{a}^{b}\frac{x{\rm v'}(x)}{\sqrt{(b-x)(x-a)}}dx=2\pi{n}.
\eea
These can be found, for instance, in \cite{YIsmail1997, YN1998}.
From (\ref{a2a}), (\ref{a3}) and integral formulas in the Appendix A, it is found that
$a$ and $b$ satisfy the following algebraic equations,
\bea\label{a7}
\frac{a+b}{2(ab)^{\frac{3}{2}}}t+\frac{\alpha}{\sqrt{ab}}-\frac{\beta}{\sqrt{(1-a)(1-b)}}=0,
\eea
and
\bea\label{a8}
\frac{t}{\sqrt{ab}}-\frac{\beta}{\sqrt{(1-a)(1-b)}}+2n+\alpha+\beta=0.
\eea
Here the parameters $a$ and $b$ determines the end points of the support of the equilibrium density.
\par
Let $X:=1/\sqrt{ab},$ and eliminating $a+b$ from (\ref{a7}) and (\ref{a8}), then $X$ satisfies the quintic,
\bea\label{a10}
\frac{X^{3}t}{2}-\frac{\beta^{2}X^{3}t}{2(tX+2n+\alpha+\beta)^{2}}+\alpha{X}-\frac{Xt}{2}-(2n+\alpha+\beta)=0. \nonumber
\eea
We now state a theorem which describes the large $n$ asymptotic of $P_n(0;t,\alpha,\beta)$, without displaying the detail steps involved, since these are
quite straightforward.
\par

\begin{theorem}
If ${\rm v}(x)=-\ln{w(x)}=\frac{t}{x}-\alpha{\ln{x}}-\beta{\ln(1-x)}$, $x \in [0, 1]$, $t \geq 0,$ $\alpha>0,$ $\beta >0,$ the evaluation at $z=0$ of $S_{1}(z; t, \alpha, \beta)$, $S_{2}(z; t, \alpha, \beta)$, and $P_{n}(z; t, \alpha, \beta)$ given by
\begin{equation}
{\rm exp}\left[-S_{1}(0; t, \alpha, \beta)\right] \sim 2^{-1}(2n+\alpha+\beta)^{\frac{1}{2}}\{2^{-1}t(2n+\alpha+\beta)^{2}\}^{-\frac{1}{6}},
\end{equation}
\begin{align}\label{a117}
{\rm exp}\left[-S_{2}(0; t, \alpha, \beta)\right] \sim &(-1)^{n}4^{-n}(2n+\alpha+\beta)^{\alpha}{\rm exp}\left[\{2^{-1}t(2n+\alpha+\beta)^{2}\}^{\frac{1}{3}}-(2\alpha+\beta){\ln2}\right. \nonumber \\
&\left.+\frac{1}{2}\{2^{-1}t(2n+\alpha+\beta)^{2}\}^{\frac{1}{3}}-\frac{\alpha}{3}\ln{\left(2^{-1}t(2n+\alpha+\beta)^{2}\right)}\right],
\end{align}
and
\begin{align}\label{a118}
P_{n}(0; t, \alpha, \beta) &\sim {\rm exp}\left[-S_{1}(0; t, \alpha, \beta)-S_{2}(0; t, \alpha, \beta)\right] \nonumber \\
&\sim (-1)^{n}4^{-n}n^{\alpha+\frac{1}{2}}2^{-(\beta+\alpha+\frac{1}{2})}{\rm exp}\left[\frac{3}{2}\{2^{-1}t(2n+\alpha+\beta)^{2}\}^{\frac{1}{3}}-\frac{\alpha}{3}\ln{\{2^{-1}t(2n+\alpha+\beta)^{2}\}}\right. \nonumber \\
&\left.-\frac{1}{6}\ln{\{2^{-1}t(2n+\alpha+\beta)^{2}\}}\right].
\end{align}
The above asymptotic estimations are uniform with respect to $t\in(0,t_{0}],$ $0<t_{0}<\infty,$ $\alpha>0,$ $\beta>0,$ $n\rightarrow\infty$ such that $n^{2}t$ is fixed.
\end{theorem}
\par
In order to derive the constant $c_{1}(\alpha)$ in $(\ref{a134})$, we still need to obtain $P_{n}(0; 0, \alpha, \beta)$, which is the constant terms of monic polynomials orthogonal
with respect to the ''shifted'' Jacobi weight $w(x;0,\alpha,\beta)=x^{\alpha}(1-x)^{\beta},\;x\in[0,1].$
This can be found from the monic polynomial orthogonal with respect to  Jacobi weight $w(x)=(1-x)^{\alpha}(1+x)^{\beta},$ $x \in [-1, 1].$
Taking a result from \cite{CI2005}, we find that,
\begin{equation}\label{a119}
P_{n}(0;0,\alpha,\beta)=(-1)^{n}\frac{\Gamma(n+1+\alpha)\Gamma(n+1+\alpha+\beta)}{\Gamma(\alpha+1)\Gamma(2n+1+\alpha+\beta)}\sim \frac{(-1)^{n}4^{-n}n^{\alpha+\frac{1}{2}}2^{-(\alpha+\beta+\frac{1}{2})}\sqrt{2\pi}}{\Gamma(\alpha+1)}, \nonumber
\end{equation}
\begin{equation*}
P_{n}(1;0,\alpha,\beta)=\frac{\Gamma(n+1+\beta)\Gamma(n+1+\alpha+\beta)}{\Gamma(\beta+1)\Gamma(2n+1+\alpha+\beta)}\sim \frac{4^{-n}n^{\beta+\frac{1}{2}}2^{-(\alpha+\beta+\frac{1}{2})}\sqrt{2\pi}}{\Gamma(\beta+1)},
\end{equation*}
where use has been made of the asymptotic formula
\begin{equation*}
\Gamma(n+1+\alpha)\sim \sqrt{2\pi}n^{n+\alpha+\frac{1}{2}}e^{-n} \quad {\rm as}, \quad n\rightarrow\infty.
\end{equation*}

{\bf Remark 3:}
To derive the constant $c_{1}(\alpha)$ in (\ref{a134}), we recall the asymptotic estimation of $P_{n}(0; t, \alpha, \beta)$, (\ref{a118}), as
\begin{align}\label{a135}
P_{n}(0; t, \alpha, \beta) &\sim {\rm exp}\left[-S_{1}(0; t, \alpha, \beta)-S_{2}(0; t, \alpha, \beta)\right] \nonumber \\
&\sim \frac{(-1)^{n}4^{-n}n^{\alpha+\frac{1}{2}}2^{-(\alpha+\beta+\frac{1}{2})}\sqrt{2\pi}}{\Gamma(\alpha+1)}\cdot\frac{\Gamma(\alpha+1)}{\sqrt{2\pi}}{\rm exp}\left[\frac{3}{2}\{2^{-1}t(2n+\alpha+\beta)^{2}\}^{\frac{1}{3}}\right.\nonumber\\
&\left.-\frac{\alpha}{3}\ln{\{2^{-1}t(2n+\alpha+\beta)^{2}\}}-\frac{1}{6}\ln{\{2^{-1}t(2n+\alpha+\beta)^{2}\}}\right]\nonumber\\
&\sim \frac{(-1)^{n}4^{-n}n^{\alpha+\frac{1}{2}}2^{-(\alpha+\beta+\frac{1}{2})}\sqrt{2\pi}}{\Gamma(\alpha+1)}{\rm exp}\left[\ln{\frac{\Gamma(\alpha+1)}{\sqrt{2\pi}}}+\frac{3}{2}\{2^{-1}t(2n+\alpha+\beta)^{2}\}^{\frac{1}{3}}\right.\nonumber\\
&\left.-\frac{\alpha}{3}\ln{\{2^{-1}t(2n+\alpha+\beta)^{2}\}}-\frac{1}{6}\ln{\{2^{-1}t(2n+\alpha+\beta)^{2}\}}\right],
\end{align}
Hence,
\begin{align*}
\frac{P_{n}(0; t, \alpha, \beta)}{P_{n}(0; 0, \alpha, \beta)}\sim &{\rm exp}\left(\frac{3}{2}\{2^{-1}t(2n+\alpha+\beta)^{2}\}^{\frac{1}{3}}-\frac{\alpha}{3}\ln{\{2^{-1}t(2n+\alpha+\beta)^{2}\}}\right.\\
&\left.-\frac{1}{6}\ln{\{2^{-1}t(2n+\alpha+\beta)^{2}\}}+c_{1}(\alpha)\right),
\end{align*}
from which $c_{1}(\alpha)$ is found to be
$$\ln{\frac{\Gamma(\alpha+1)}{\sqrt{2\pi}}}.$$
Replacing $2^{-1}t(2n+\alpha+\beta)^{2}$ by $s$, then (\ref{a134}) follows.
\\
By (\ref{aha1}) one finds
$$
c(\alpha+1)-c(\alpha)=c_{1}(\alpha)=\ln{\frac{\Gamma(\alpha+1)}{\sqrt{2\pi}}},
$$
giving
$$
c(\alpha)=\ln\frac{G(\alpha+1)}{(2\pi)^{\frac{\alpha}{2}}},
$$
which we recognize to be the Tracy-Widom constant that appeared in the Bessel kernel problem.

Here $G(z)$ is the Barnes-G function.

\subsection{An algorithm of the Hankel determinant for small $s$.}
In this subsection, we use the formulas derived by Normand in \cite{JMN2004} and similar process in \cite{PFW2006} to compute the Hankel determinant. We look for expansion of
 the Hankel determinant around $t=0$. We scale the variable $t$ as $s=2n^{2}t$ in the expansion and let $n \rightarrow \infty$, in this way
  we verify the asymptotic expansion for small $s$ derived in section 2.
\par
Recall that
\bea\label{a98}
D_{n}(t,\alpha,\beta)={\rm det}\left[\mu_{i+j}(t)\right]_{i,j=0}^{n-1}, \nonumber
\eea
where the moments $\mu_{i+j}(t)$ are given by
\bea\label{a99}
\mu_{i+j}(t)=\int_{0}^{1}x^{i+j}x^{\alpha}(1-x)^{\beta}e^{-\frac{t}{x}}dx, \quad i,j=0,1, \ldots \nonumber
\eea
Although the moments $\mu_{i+j}(t)$ depending on $\alpha$ and $\beta$, we do not display this to lighten notations.
\begin{lemma}
The moments  $\mu_{m}(t)$ are,
\bea\label{aa101}
\mu_{m}(t)=\int_{0}^{1}x^{m}x^{\alpha}(1-x)^{\beta}e^{-\frac{t}{x}}dx=\varphi_{m}(t)+t^{m+\alpha+1}\psi_{m}(t),\quad m \in \mathbb{N},
\eea
where $\varphi_{m}(t)$ and $\psi_{m}(t)$ are analytic at $t=0$,
$$
\varphi_{m}(t)=\frac{e^{-t}\Gamma(\beta+1)\Gamma(m+\alpha+1)}{\Gamma(m+2+\alpha+\beta)}\ _{1}\mathrm{F}_{1}(1+\beta,-m-\alpha,t),
$$
and,
$$
\psi_{m}(t)=e^{-t}\Gamma(-m-\alpha-1)\ _{1}\mathrm{F}_{1}(m+2+\alpha+\beta,m+2+\alpha,t),
$$
where $_1\mathrm{F}_1(a,b;z)$ is the confluent hypergeometric function of the first kind.

Futrthermore, $\mu_{m}(t)$ has series expansion around $t=0$,
\begin{align}\label{aa102}
\mu_{m}(t)&=\varphi_{m}(0)+t\varphi_{m}'(0)+\frac{t^{2}}{2!}\varphi_{m}''(0)+\mathcal{O}(t^{3})+t^{m+\alpha+\lambda+1}\psi_{m}(0)(1+\mathcal{O}(t^{2}))\nonumber\\
&+\mathcal{O}\left(t^{2(m+\alpha+\lambda+1)}\right),
\end{align}
where
\bea\label{aa103}
\varphi_{m}(0)=\frac{\Gamma(\beta+1)\Gamma(m+\alpha+1)}{\Gamma(m+2+\alpha+\beta)},\quad \varphi_{m}'(0)=-\frac{\Gamma(\beta+1)\Gamma(m+\alpha)}{\Gamma(m+1+\alpha+\beta)},
\eea

\bea\label{aa104}
\varphi_{m}''(0)=\frac{\Gamma(\beta+1)\Gamma(m+\alpha-1)}{\Gamma(m+\alpha+\beta)},
\eea
and
\bea\label{a105}
\psi_{m}(0)=\Gamma(-m-\alpha-1). \nonumber
\eea
The expansions are valid for $\alpha>0,$ $\beta>0,$ $m \in \mathbb{N}$ and $|\arg{t}|<\pi.$
\end{lemma}

\begin{proof}
From (\cite{GR2007}, P367), we find
$$
\mu_{m}(t)=\int_{0}^{1}x^{m}x^{\alpha}(1-x)^{\beta}e^{-\frac{t}{x}}dx=t^{\frac{m+\alpha}{2}}e^{-\frac{t}{2}}\Gamma(\beta+1)W\left(-\frac{m+2+\alpha+2\beta}{2},\frac{m+1+\alpha}{2},t\right),
$$
where $W$ denotes the Whittaker function
\begin{align*}
W(\lambda,\mu,z)=&\frac{\Gamma(-2\mu)}{\Gamma(\frac{1}{2}-\mu-\lambda)}z^{\mu+\frac{1}{2}}e^{-\frac{z}{2}}\ _{1}\mathrm{F}_{1}(\mu-\lambda+\frac{1}{2},2\mu+1,z)\\
&+\frac{\Gamma(2\mu)}{\Gamma(\frac{1}{2}+\mu-\lambda)}z^{-\mu+\frac{1}{2}}e^{-\frac{z}{2}}\ _{1}\mathrm{F}_{1}(-\mu-\lambda+\frac{1}{2},-2\mu+1,z),
\end{align*}
subject to $|\arg{z}|<\pi$ (\cite{GR2007}, P1023). Hence $(\ref{aa101})$ and $(\ref{aa102})$ follow.

\end{proof}

\begin{theorem}
The series expansion of the Hankel determinant is given by,
\begin{align}\label{a109}
&D_{n}(t,\alpha,\beta)={\rm det}\left[\mu_{k+j}(t)\right]_{k,j=0}^{n-1}=D_{n}(0,\alpha,\beta)\left[1-\frac{n(n+\alpha+\beta)t}{\alpha}\right.\nonumber\\ 
&\left.+\frac{n(n+\alpha+\beta)(n(n+\alpha+\beta)\alpha+\beta)t^{2}}{2\alpha(\alpha^{2}-1)}+\mathcal{O}\left(t^{3}\right)\right.\nonumber\\
&\left.+\frac{\Gamma(n+\alpha+1)\Gamma(n+1+\alpha+\beta)\pi{t^{\alpha+1}}}{\Gamma(\beta+1)\Gamma(\alpha+1)\Gamma^{2}(\alpha+2){\rm sin}(\alpha{\pi})(n-1)!\prod_{j=1}^{n-1}(\beta+j)}\left(1+\mathcal{O}(t)\right)+\mathcal{O}\left(t^{2(\alpha+1)}\right)\right],
\end{align}
subject to $\alpha \notin \mathbb{Z},$ $\alpha>0$ and $|\arg{t}|<\pi.$ $D_{n}(0,\alpha,\beta)$ has a closed form expression \cite{EC2005},
\begin{align}\label{b28}
D_{n}(0,\alpha,\beta)=4^{-n(n+\alpha+\beta)}&(2\pi)^{n}\frac{\Gamma(\frac{\alpha+\beta+1}{2})G^{2}(\frac{\alpha+\beta+1}{2})G^{2}
(\frac{\alpha+\beta}{2}+1)}{G(\alpha+\beta+1)G(\alpha+1)G(\beta+1)}\nonumber\\
&\times \frac{G(n+1)G(n+\alpha+1)G(n+\beta+1)G(n+\alpha+\beta+1)}{G^{2}(n+\frac{\alpha+\beta+1}{2})G^{2}(n+1+\frac{\alpha+\beta}{2})\Gamma(n+\frac{\alpha+\beta+1}{2})},
\end{align}
where $G(z)$ is the Barnes $G$-function.
\end{theorem}

\begin{proof}
By $(\ref{aa102})$,
\begin{align}\label{a103}
D_{n}(t, \alpha, \beta)=&{\rm det}\left[\mu_{k+j}(t)\right]_{k,j=0}^{n-1} \sim {\rm det}\left[\varphi_{k+j}(0)+t\varphi_{k+j}'(0)+t^{2}\frac{\varphi_{k+j}''(0)}{2!}+t^{k+j+\alpha+1}\psi_{k+j}(0)\right]_{k,j=0}^{n-1}\nonumber\\
&\sim {\rm det}\left[\varphi_{k+j}(0)\right]_{k,j=0}^{n-1}+t<t>{\rm det}\left[\varphi_{k+j}(0)+t\varphi_{k+j}'(0)\right]_{k,j=0}^{n-1}\nonumber\\
&+t^{2}<t^{2}>{\rm det}\left[\varphi_{k+j}(0)+t\varphi_{k+j}'(0)+t^{2}\frac{\varphi_{k+j}''(0)}{2!}\right]_{k,j=0}^{n-1}\nonumber\\
&+t^{\alpha+1}<t^{\alpha+1}>{\rm det}\left[\varphi_{k+j}(0)+t^{k+j+\alpha+1}\psi_{k+j}(0)\right]_{k,j=0}^{n-1}.
\end{align}
Here $<t^{m}>f(t)$ denotes the coefficient of $t^{m}$ in the series expansion of $f(t)$ in $t.$ We note here Normand's formula
\bea\label{a101}
{\rm det}\left[\frac{\Gamma(z_{j}+i)}{\Gamma(az_{j}+b+i)}\right]_{i,j=0}^{n-1}=\prod_{j=0}^{n-1}\frac{(b+(n-1-j)(1-a))_{j}\Gamma(z_{j})}{\Gamma(az_{j}+b+n-1)}\prod_{0\leq i<j \leq N-1}(z_{j}-z_{i}), \nonumber
\eea
which can be found in \cite{JMN2004}, and
\bea\label{a102}
{\rm det}\left[\mathbf{a}_{1},\cdots,{\mathbf{a}_{j}+\mathbf{b}_{j}},\cdots{\mathbf{a}_{n}}\right]={\rm det}\left[\mathbf{a}_{1},\cdots,{\mathbf{a}_{j}},\cdots,{\mathbf{a}_{n}}\right]+{\rm det}\left[\mathbf{a}_{1},\cdots,{\mathbf{b}_{j}},\cdots,{\mathbf{a}_{n}}\right],
\eea
where $\mathbf{a}_{j}$, $\mathbf{b}_{j}$ are column vectors. We obtain, after an extensive computations, while bearing in mind
$$
D_{n}(0,\alpha,\beta)={\rm det}\left[\varphi_{k+j}(0)\right]_{k,j=0}^{n-1},
$$
the following
\begin{align}
<t>{\rm det}\left[\varphi_{k+j}(0)+t\varphi_{k+j}'(0)\right]_{k,j=0}^{n-1}=&<t>{\rm det}\left[\frac{\Gamma(\beta+1)\Gamma(\alpha+1+j+k)}{\Gamma(2+\alpha+\beta+j+k)}\right.\nonumber\\
&\left.-t\frac{\Gamma(\beta+1)\Gamma(\alpha+j+k)}{\Gamma(1+\alpha+\beta+j+k)}\right]_{k,j=0}^{n-1}. \nonumber
\end{align}
With  $(\ref{a102})$, we find,
$$
<t>{\rm det}\left[\frac{\Gamma(\beta+1)\Gamma(\alpha+1+j+k)}{\Gamma(2+\alpha+\beta+j+k)}-t\frac{\Gamma(\beta+1)\Gamma(\alpha+j+k)}{\Gamma(1+\alpha+\beta+j+k)}\right]_{k,j=0}^{n-1}=-\frac{n(n+\alpha+\beta)t}{\alpha}D_{n}(0,\alpha,\beta),
$$
The terms $<t^{2}>f(t)$ and $<t^{\alpha+1}$ can be similarly derived, although with greater effort.
\end{proof}

\begin{cor}
From the series expansion of $D_{n}(t, \alpha, \beta)$ around $t=0$, see $(\ref{a109})$,
the series expansion of $H_{n}(t,\alpha,\beta)$ around $t=0$ follows,
\begin{align}\label{aa28}
H_{n}(t,\alpha,\beta)&=t\frac{d}{dt}\ln{\frac{D_{n}(t,\alpha,\beta)}{D_{n}(0,\alpha,\beta)}}
=-\frac{n(n+\alpha+\beta)}{\alpha}t+\frac{n(n+\alpha+\beta)(n(n+\alpha+\beta)+\alpha\beta)}{\alpha^{2}(\alpha^{2}-1)}t^{2}\nonumber\\ &+\mathcal{O}(t^{3})+t^{\alpha+1}\frac{\Gamma(n+\alpha+1)\Gamma(n+1+\alpha+\beta)\pi}{(\alpha+1)^{2}\Gamma^{3}(\alpha+1){\rm sin}(\alpha{\pi})\Gamma(n)\Gamma(n+\beta)}\left(1+\mathcal{O}(t)\right)+\mathcal{O}\left(t^{2(\alpha+1)}\right),
\end{align}
subject to $\alpha \notin \mathbb{Z},$ $\alpha>0,$ $t>0$ and $|\arg{t}|<\pi.$
\end{cor}
\begin{proof}
With $H_{n}(t, \alpha, \beta)$ given by $(\ref{b42})$, the series expansion of $D_{n}(t, \alpha, \beta)$, equation $(\ref{a109})$,
and the identity
$$
\frac{\Gamma(n+\beta)}{\Gamma(\beta+1)}=\prod_{j=1}^{n-1}(\beta+j),
$$
the equation $(\ref{aa28})$ is obtained.
\end{proof}

\begin{cor}
Sending $t\rightarrow 0^{+},$ $n\rightarrow \infty,$ and $s=2n^{2}t$ such that $s\in(0,\infty)$ is finite, we have,
\begin{align}\label{aa29}
{\cal H}(s,\alpha,\beta)=&\lim_{n \rightarrow \infty}s\frac{d}{ds}\ln{\frac{D_{n}(\frac{s}{2n^{2}},\alpha,\beta)}{{D_{n}(0,\alpha,\beta)}}}
=s\frac{d}{ds}\ln{\Delta(s,\alpha,\beta)}=-\frac{s}{2\alpha}+\frac{s^{2}}{4\alpha^{2}(\alpha^{2}-1)}+\mathcal{O}(s^{3})\nonumber\\
&+\frac{{\pi}s^{\alpha+1}}{2^{\alpha+1}(\alpha+1)^{2}\Gamma^{3}(\alpha+1){\rm sin}(\alpha{\pi})}(1+\mathcal{O}(s))+\mathcal{O}(s^{2(\alpha+1)}),
\end{align}
subject to $\alpha \notin \mathbb{Z},$ $\alpha>0$ and $|\arg{s}|<\pi.$
\end{cor}
\begin{proof}
Substituting $s=2n^{2}t$ into $(\ref{aa28}),$ and taking limit $n \rightarrow \infty$, then $(\ref{aa29})$ is obtained with
\begin{align*}
\lim_{n \rightarrow \infty}\frac{s^{\alpha+1}}{2^{\alpha+1}n^{2(\alpha+1)}}\frac{\Gamma(n+\alpha+1)\Gamma(n+1+\alpha+\beta)\pi}{(\alpha+1)^{2}\Gamma^{3}(\alpha+1){\rm sin}(\alpha{\pi})\Gamma(n)\Gamma(n+\beta)}=\frac{{\pi}s^{\alpha+1}}{2^{\alpha+1}(\alpha+1)^{2}\Gamma^{3}(\alpha+1){\rm sin}(\alpha{\pi})},
\end{align*}
where $\alpha \notin \mathbb{Z},$ $\alpha>0$ and $\Gamma(n+\alpha+1) \sim \Gamma(n+1)n^{\alpha},$ as $n\rightarrow \infty,$ is used in the proof.
\end{proof}

{\bf Remark.}
Note that the small $s$ expansion of  ${\cal H}(s,\alpha,\beta)$ is independent of $\beta$ and coincides with the characterization of Painlev\'{e} equation in the Theorem 3.
In particular, the ``first portion"  of $(\ref{aa29})$ is the same as the series expansion,
\begin{align*}
{\cal H}(s,\alpha,\beta)&=s\frac{d}{ds}\ln{\Delta(s,\alpha,\beta)}=\frac{\left(sg'-g\right)^{2}}{4g^{2}}+\frac{4{\alpha}sg-s^{2}}{16g^{2}}-g-\frac{\alpha^{2}}{4}\\
&=-\frac{1}{2\alpha}s+\frac{1}{4\alpha^{2}(\alpha^{2}-1)}s^{2}-\frac{1}{2\alpha^{3}(\alpha^{2}-4)(\alpha^{2}-1)}s^{3}+\mathcal{O}({s^{4}}),
\end{align*}
where the small expansion $s$ of $g(s,\alpha,\beta)$ is given by $(\ref{a72})$. We note that the ``special function portion" of $(\ref{aa29})$, cannot be obtained through power series.
 This method,  proposed by Normand,  capture both ``special functions" and  ``power series". T complexity in the computations increases very rapidly, as one attempt to
 include term with higher powers of $s$ and should be a worthy future project to explore.

\section{Double scaling analysis \uppercase\expandafter{\romannumeral2}.}
In this section, we propose another double scaling scheme by setting $\alpha=-2n+\lambda,$ $s=nt,$ as $n\rightarrow \infty,$ $t\rightarrow 0$ such that $\lambda$ and $s$ are finite, the scaled determinant has an integral representation in terms of a particular ${\rm P_{V}},$ which can be reduced to a ${\rm P_{III}},$
from which its expansions for small and large $s$ can be found.

\subsection{The Painlev\'{e} equations.}
For convenience, we introduce a number of items;
\begin{equation}\label{A1}
t\frac{d}{dt}\ln{\frac{\Delta_{n}(t,\alpha,\beta)}{\Delta_{n}(0,\alpha,\beta)}}:=t\frac{d}{dt}\ln{\frac{D_{n}(t,\alpha,\beta)}{D_{n}(0,\alpha,\beta)}}+n^{2}-n\lambda,\;\;{\rm and\;\;} \widehat{H}_{n}(t,\alpha,\beta):=H_{n}(t,\alpha,\beta)+n^{2}-n\lambda,
\end{equation}
where
$$
H_{n}(t,\alpha,\beta):=t\frac{d}{dt}\ln{\frac{D_{n}(t,\alpha,\beta)}{D_{n}(0,\alpha,\beta)}}.
$$
\\
We proceed as follows.
\\
Substituting $\alpha=-2n+\lambda$ into (\ref{A1}), we find,
$$
\Delta_{n}(t,\lambda,\beta)=\frac{t^{n^{2}-n\lambda}}{n!}D_{n}(t,\lambda,\beta)=\frac{t^{n^{2}-n\lambda}}{n!}\int_{(0,1)^{n}}\prod_{1\leq j<k\leq n}\left(x_{j}-x_{k}\right)^{2}\prod_{i=1}^{n}x_{l}^{-2n+\lambda}\left(1-x_{l}\right)^{\beta}e^{-\frac{t}{x_{l}}}dx_{l}.
$$
With the change of variable
$
z=\frac{(1-x)t}{x},
$
mapping the interval $[0,1]$ to $[0,\infty]$, one finds,
\bea\label{c01}
\widehat{D}_{n}(t,\lambda,\beta)=\frac{(-1)^{n}}{n!}\int_{(0,\infty)^{n}}\prod_{1\leq j<k\leq n}\left(z_{k}-z_{j}\right)^{2}\prod_{\ell=1}^{n}(z_{\ell}+t)^{-(\lambda+\beta)}z_{\ell}^{\beta}e^{-z_{\ell}-t}dz_{\ell},
\eea
and
\begin{equation*}
(-1)^{n}\Delta_{n}(0,\lambda)=\frac{G(n+1)G(n+1-\lambda)}{G(-\lambda)},
\end{equation*}
where $\lambda<0.$

This leads us to study the Hankel determinant originates generated by the weight
$$
w(x;t)=(x+t)^{-(\lambda+\beta)}x^{\beta}{\rm e}^{-x-t}, \quad \lambda<0,\quad \beta>0, \quad x\in (0, \infty).
$$
We would like to consider a more general weight of the form,
\begin{equation}\label{DD36}
w(x;t)=(x+t)^{-(\lambda+\beta)}x^{\gamma}e^{-x-t}, \quad \lambda<0, \beta>0, \gamma>0,\gamma-\lambda-\beta>-1, x\in (0, \infty),
\end{equation}
which the reduces to the one above (\ref{DD36}), if $\gamma=\beta$ and degenerates to the Laguerre weight at $t=0.$
\par
We introduce two quantities in term of the monic polynomials $P_{n}(x)$ orthogonal with respect to (\ref{DD36}), over $[0,\infty).$
These are
\begin{equation}\label{D9}
\widehat{R}_{n}(t):=\frac{\gamma}{h_{n}}\int_{0}^{\infty}\frac{P_{n}^{2}(y)}{y}w(y;t)dy,
\end{equation}
\begin{equation}\label{D10}
\widehat{r}_{n}(t):=\frac{\gamma}{h_{n-1}}\int_{0}^{\infty}\frac{P_{n-1}(y)P_{n}(y)}{y}w(y;t)dy.
\end{equation}
Moreover, we find $\widehat{r}_{n}$ and $\widehat{R}_{n}$ satisfy the following Riccati equations,
\begin{equation}\label{D26}
t\frac{d}{dt}\widehat{r}_{n}=\left(2n-\lambda-\beta+\gamma\right)\widehat{r}_{n}+n(n-\lambda-\beta)+\frac{\widehat{r}_{n}(\widehat{r}_{n}-\gamma)}
{\widehat{R}_{n}}
-\frac{\widehat{r}_{n}(\widehat{r}_{n}+2n-\lambda-\beta)+n(n-\lambda-\beta)}{1-\widehat{R}_{n}},
\end{equation}
\begin{equation}\label{D27}
t\frac{d}{dt}\widehat{R}_{n}=2\widehat{r}_{n}-\gamma+\left(2n+\gamma-\beta-\lambda+t(\widehat{R}_{n}-1)\right)\widehat{R}_{n}.
\end{equation}

Let
\begin{equation}\label{D25}
\widehat{H}_{n}(t,\gamma,\lambda,\beta)
:=t\frac{d}{dt}\ln{\frac{\widehat{D}_{n}(t,\gamma,\lambda,\beta)}{\widehat{D}_{n}(0,\gamma,\lambda,\beta)}}
-t\sum_{j=0}^{n-1}R_{j},
\end{equation}
where the derivation of $-t\sum_{j=0}^{n-1}R_j(t),$ obtained through ladder operators, is not re-produce here.
\\
Note that,
\begin{equation*}
(-1)^{n}\Delta_{n}(0,\gamma,\lambda,\beta)=\frac{G(n+1)G(n+1+\gamma-\beta-\lambda)}{G(\gamma-\beta-\lambda)}.
\end{equation*}
and a simply identity
\begin{equation}\label{D30}
\widehat{r}_{n}(t)=\widehat{H}_{n}'(t).
\end{equation}

\begin{theorem}
For finite $n$, let
$$
\widehat{R}_{n}(t)=:1+\frac{1}{\widehat{y}(t,\gamma,\lambda,\beta)-1},
$$
then $\widehat{y}(t,\gamma,\lambda,\beta)$ satisfies
{\small
\begin{equation}\label{D28}
\widehat{y}''=\frac{(3\widehat{y}-1)\left(\widehat{y}'\right)^{2}}{2\widehat{y}(\widehat{y}-1)}
-\frac{\widehat{y}'}{t}+\frac{\left(\widehat{y}-1\right)^{2}}{2t^{2}}\left((\beta+\lambda)^{2}\widehat{y}-\frac{\gamma^{2}}{\widehat{y}}\right)-\frac{\left(2n+1+\gamma-\beta-\lambda\right)\widehat{y}}{t}-\frac{\widehat{y}\left(\widehat{y}+1\right)}{2(\widehat{y}-1)},
\end{equation}
}
with the initial conditions
$$
\widehat{y}(0,\gamma,\lambda,\beta)=\frac{\gamma}{\beta+\lambda}, \quad \widehat{y}'(0,\gamma,\lambda,\beta)=\frac{\gamma(2n+1+\gamma-\beta-\lambda)}{(\beta+\lambda)((\beta+\lambda-\gamma)^{2}-1)}.
$$
The quantity $\widehat{H}_{n}(t,\gamma,\lambda,\beta)$ defined in $(\ref{D25}),$ and satisfies the following second order ODE,
\begin{equation}\label{D29}
\left(t\widehat{H}_{n}''\right)^{2}=4\left(n+\widehat{H}_{n}'\right)\left(\widehat{H}_{n}-t\widehat{H}_{n}'\right)
\left(\widehat{H}_{n}'-\gamma\right)+\left((t-\gamma+\beta+\lambda)\widehat{H}_{n}'-\widehat{H}_{n}-n\gamma\right)^{2},
\end{equation}
with the initial conditions
$$
\widehat{H}_{n}(0,\gamma,\lambda,\beta)=0, \quad \widehat{H}_{n}'(0,\gamma,\lambda,\beta)=\frac{n\gamma}{\beta+\lambda-\gamma}.
$$
Moreover,
\begin{equation}\label{D31}
\widehat{H}_{n}(t,\gamma,\lambda,\beta):=\frac{\left(t\widehat{y}'\right)^{2}}{4\widehat{y}(\widehat{y}-1)^{2}}
-\frac{\left((\beta+\lambda)\widehat{y}-\gamma\right)^{2}}{4\widehat{y}}+\frac{t(\beta+\lambda-2n)\widehat{y}}{2(\widehat{y}-1)}
-\frac{\gamma{t}}{2(\widehat{y}-1)}-\frac{t^{2}\widehat{y}}{4(\widehat{y}-1)^2}.
\end{equation}
\end{theorem}

\begin{proof}
Eliminating $\widehat{r}_{n}$ from $(\ref{D26})$ and $(\ref{D27})$, substituting $\widehat{R}_{n}(t)=1+1/(\widehat{y}(t)-1)$ into the resulting equation, we see that $\widehat{y}(t)$ satisfies $(\ref{D28}).$
\\
Eliminating $\hat{r}_n$ from  $(\ref{D30})$ and $(\ref{D26}),$ we obtain the following
\begin{equation*}
\widehat{R}_{n}=F(\widehat{H}_{n},\widehat{H}_{n}',\widehat{H}_{n}''),
\end{equation*}
where $F(\cdot,\cdot,\cdot)$  is a function of three variables, which we do not display.
\\
Combining the expression for $R_n,$ $(\ref{D25}),$ $(\ref{D30})$, and
\begin{equation}\label{D20}
\sum_{j=0}^{n-1}\widehat{R}_{j}=\frac{\widehat{R}_{n}\left(n(n-\lambda)+(\gamma-\beta)n
+t\widehat{r}_{n}\right)-\widehat{r}_{n}(t+\lambda+\beta-\gamma-2n)-n\gamma}{t(1-\widehat{R}_{n})}
+\frac{\widehat{r}_{n}(\widehat{r}_{n}-\gamma)}{t\widehat{R}_{n}(1-\widehat{R}_{n})},
\end{equation}
we arrive at the non-linear second order ODE which is satisfied by $\widehat{H}_{n}(t,\gamma,\lambda,\beta).$
\par
Finally, combining $(\ref{D20})$ and $(\ref{D25})$, we see that,
\begin{equation*}
\widehat{H}_{n}(t,\gamma,\lambda,\beta)=-\frac{\widehat{R}_{n}\left(n(n-\lambda)+(\gamma-\beta)n+t\widehat{r}_{n}\right)
-\widehat{r}_{n}(t+\lambda+\beta-\gamma-2n)-n\gamma}{(1-\widehat{R}_{n})}-\frac{\widehat{r}_{n}(\widehat{r}_{n}-\gamma)}{\widehat{R}_{n}
(1-\widehat{R}_{n})}.
\end{equation*}
The equation $(\ref{D31})$ is found, with $(\ref{D27})$ and $\widehat{R}_{n}(t)=1+1/(\widehat{y}(t)-1)$.
\end{proof}

{\bf Remark 4:}
If $\widehat{H}_{n}(t,\gamma,\lambda,\beta)\rightarrow \widehat{H}_{n}(t,\gamma,\lambda,\beta)-nt-n(\beta+\lambda)$, then the non-linear second ODE (\ref{D29}) becomes,
\begin{equation*}
\left(t\widehat{H}_{n}''\right)^{2}=\left(\widehat{H}_{n}-t\widehat{H}_{n}'+2\right(\widehat{H}_{n}'\left)^{2}-\delta_{n}\widehat{H}_{n}'\right)^{2}-4\widehat{H}_{n}'\left(\widehat{H}_{n}'-n\right)\left(\widehat{H}_{n}'-\beta-\lambda\right)\left(\widehat{H}_{n}'-n-\gamma\right),
\end{equation*}
where $\delta_{n}=2n+\gamma+\beta+\lambda.$ The equation above is the Jimbo-Miwa-Okamoto \cite{JM1981,O1981} $\sigma$-form of
${\rm P_{V}}$, with parameters 
$$
v_{0}=0,\quad v_{1}=-n, \quad v_{2}=-\beta-\lambda, \quad v_{3}=-n-\gamma.
$$
Similar to the investigation in previous sections, let,
\begin{align}\label{D35}
\widehat{{\cal H}}\left(s,\gamma\right):=\lim_{n \rightarrow \infty}\widehat{H}_{n}\left(\frac{s}{n},\gamma,\lambda,\beta\right),  \quad \Delta(s,\gamma)=\lim_{n \rightarrow \infty}\frac{\Delta_{n}\left(\frac{s}{n},\gamma,\lambda,\beta\right)}{\Delta_{n}\left(0,\gamma,\lambda,\beta\right)},\quad \widehat{g}(s,\gamma):=\lim_{n \rightarrow \infty}\widehat{y}\left(\frac{s}{n},\gamma,\lambda,\beta\right),
\end{align}
we have,
\begin{theorem}
Let $s:=nt,$ $t \rightarrow 0^{+},$ $n\rightarrow \infty,$ such that $s$ is finite. Then $\widehat{g}(s,\gamma)$ satisfies
\begin{equation}\label{D32}
\widehat{g}''=\frac{3\widehat{g}-1}{2\widehat{g}(\widehat{g}-1)}\left(\widehat{g}'\right)^{2}-\frac{\widehat{g}'}{s}
+\frac{(\beta+\lambda)^{2}\widehat{g}\left(\widehat{g}-1\right)^{2}}{2s^{2}}-\frac{\gamma^{2}\left(\widehat{g}-1\right)^{2}}
{2s^{2}\widehat{g}}-\frac{2\widehat{g}}{s},
\end{equation}
with the initial data
$$
\widehat{g}(0,\gamma)=\frac{\gamma}{\beta+\lambda}, \quad \widehat{g}'(0,\gamma)=\frac{2\gamma}{(\beta+\lambda)((\beta+\lambda-\gamma)^{2}-1)}.
$$
The quantity $\widehat{{\cal H}}(s,\gamma)$ satisfies a non-linear second order ODE,
\begin{equation}\label{D33}
\left(s\widehat{{\cal H}}''\right)^{2}=4{\widehat{{\cal H}}'}\left(1+\widehat{{\cal H}}'\right)\left({\widehat{{\cal H}}}-s\widehat{{\cal H}}'\right)+\left(\left(-\gamma+\beta+\lambda\right){\widehat{{\cal H}}'}-\gamma\right)^{2},
\end{equation}
with the initial conditions
$$
\widehat{{\cal H}}(0,\gamma)=0,\quad \widehat{{\cal H}}'(0,\gamma)=\frac{\gamma}{\beta+\lambda-\gamma}.
$$
Moreover,
\begin{equation}\label{D34}
\widehat{{\cal H}}=\frac{(s\widehat{y}')^{2}}{4\widehat{y}\left(\widehat{y}-1\right)^{2}}-\frac{\left((\beta+\lambda)\widehat{y}-\gamma\right)^{2}}{4\widehat{y}}-\frac{s\widehat{y}}{\widehat{y}-1}.
\end{equation}
The equation (\ref{D32}) is satisfied by ${\rm P}_V((\beta+\lambda)^2/2,-\gamma^2/2,-2,0)).$
\end{theorem}
\begin{proof}
Inserting $s=nt$ into $(\ref{D28}),$ $(\ref{D29}),$ $(\ref{D31})$ and followed by the double scaling limit,  the equations $(\ref{D32}),$ $(\ref{D33})$ and $(\ref{D34})$ are found.
\end{proof}
{\bf Remark 5:}
If $\widehat{{\cal H}}(s,\gamma)\rightarrow \widehat{{\cal H}}(s,\gamma)-s,$ then $\widehat{{\cal H}}'(s,\gamma)\rightarrow \widehat{{\cal H}}'(s,\gamma)-1$
and the non-linear second order ODE of (\ref{D33}) becomes
\begin{equation*}
\left(s\widehat{{\cal H}}''\right)^{2}=4\widehat{{\cal H}}'\left(\widehat{{\cal H}}-s\widehat{{\cal H}}'\right)\left(\widehat{{\cal H}}'-1\right)+\left(\left(-\gamma+\beta+\lambda\right)\widehat{{\cal H}}'-\beta-\lambda\right)^2,
\end{equation*}
This turns out to be a $\sigma$-form of ${\rm P_{III}}$.

{\bf Remark 6:} Substituting $\widehat{H}_{n}(t,\lambda,\beta)=H_{n}(t,\lambda,\beta)+n^{2}-n\lambda$ and $s=nt$ into the $\sigma$-from Painlev\'{e} equation $(\ref{b2}),$ we see that
 $\widehat{{\cal H}}\left(s,\lambda,\beta\right)$ satisfies another $\sigma$-form of ${\rm P_{III}}$,
\begin{align}\label{a33}
(s\widehat{{\cal H}}^{''})^{2}=4\widehat{{\cal H}}^{'}(1+\widehat{{\cal H}}^{'})(\widehat{{\cal H}}-s\widehat{{\cal H}}^{'})+(\lambda\widehat{{\cal H}}^{'}-\beta)^{2}.
\end{align}

\subsection{The asymptotic Expansions of $P_V((\beta+\lambda)^2/2,-\gamma^2/2,-2,0)$}

For small and positive $s$, we find,
\begin{align}\label{aa13}
\widehat{g}(s)=&\frac{\gamma}{\beta+\lambda}+\frac{2\gamma}{(\beta+\lambda)((\beta+\lambda-\gamma)^{2}-1)}s\nonumber\\
&+\frac{((\beta+\lambda)^{2}-1)(3(\beta+\lambda)-5\gamma)+\gamma^{2}(\beta+\lambda+\gamma)}{(\beta+\lambda)(\beta+\lambda-\gamma)
((\beta+\lambda-\gamma)^{2}-1)
((\beta+\lambda-\gamma)^{2}-4)}s^{2}+\mathcal{O}({s^{3}}),
\end{align}
where $\beta+\lambda-\gamma \notin \mathbb{Z}.$
\par
{\bf The solution of ${\rm P_{V}\left((\beta+\lambda)^{2}/2, -\gamma^{2}/2, -2, 0\right)}$ for large $s$.}
\par
For large and positive $s,$ we find,
\begin{align}\label{4a2}
\widehat{g}(s)=&\frac{2}{\beta+\lambda}s^{\frac{1}{2}}+1+\frac{4\gamma^{2}-1}{16(\beta+\lambda)}s^{-\frac{1}{2}}+\frac{1-4\gamma^{2}}{16}s^{-1}+\frac{(4\gamma^{2}-1)(48(\beta+\lambda)^{2}-4\gamma^{2}+25)}{1024(\beta+\lambda)}s^{-\frac{3}{2}}\nonumber\\
&+\frac{(4\gamma^{2}-1)(4\gamma^{2}-8(\beta+\lambda)^{2}-17)}{256}s^{-2}+\mathcal{O}({s^{-\frac{5}{2}}}).
\end{align}
\begin{theorem}
The scaled Hankel determinant generated by the Pollaczek-Jacobi type weigh has the following small and large $s$ expansions:
\\
For small $s$,
\begin{align}\label{aa00}
\widehat{\Delta}(s,\gamma)=&\exp\left(\frac{\gamma}{\beta+\lambda-\gamma}s+\frac{\gamma(\beta+\lambda)}{2(\beta+\lambda-\gamma)^{2}((\beta+\lambda-\gamma)^{2}-1)}s^{2}\right.\nonumber\\
&\left.+\frac{2\gamma(\beta+\lambda)(\beta+\lambda+\gamma)}{3(\beta+\lambda-\gamma)^{3}((\beta+\lambda-\gamma)^{2}-4)((\beta+\lambda-\gamma)^{2}-1)}s^{3}+ \mathcal{O}({s^{4}})\right),
\end{align}
where $\beta+\lambda-\gamma \notin \mathbb{Z}.$
\\
For large $s$,
\begin{align}\label{bb00}
\widehat{\Delta}(s,\gamma)=&\exp\left(c_{2}-s-2(\beta+\lambda){s^{\frac{1}{2}}}+\frac{(\beta+\lambda)(\beta+\lambda-2\gamma)}{4}\ln{s}+\frac{(4\gamma^{2}-1)(\beta+\lambda)}{16}s^{-\frac{1}{2}}\right.\nonumber\\
&\left.-\frac{(4\gamma^{2}-1)(\beta+\lambda)^{2}}{64}s^{-1}+\mathcal{O}(s^{-\frac{3}{2}})\right),
\end{align}
where $c_{2}=c_{2}(\gamma)$ is an integration constant independent of $s$.
\end{theorem}

\begin{proof}
From $(\ref{D25}),$ $(\ref{D35})$ and $(\ref{D34})$, we obtain,
\begin{align*}
\ln{\widehat{\Delta}(s,\gamma)}&=\int_{0}^{s}\left(\frac{(\xi{\widehat{g}'(\xi)})^{2}}{4\widehat{g}(\xi)\left(\widehat{g}(\xi)-1\right)^{2}}
-\frac{\left((\beta+\lambda)\widehat{g}(\xi)-\gamma\right)^{2}}{4\widehat{g}(\xi)}
-\frac{\xi{\widehat{g}(\xi)}}{\widehat{g}(\xi)-1}\right)\frac{d\xi}{\xi}.
\end{align*}
Substituting (\ref{aa13}) and (\ref{4a2}) into the above formula, the equations (\ref{aa00}) and (\ref{bb00}) are found.
\end{proof}
At the end of this section, we evaluate the large $n$ behavior of $P_{n}(0;t,\gamma).$ In fact,
$$
(-1)^{n}P_{n}(0;t,\gamma)=\frac{D_{n}(t,\gamma+1)}{D_{n}(t,\gamma)}.
$$
For the weight,
\begin{equation*}
w(x;0)=x^{\gamma-(\lambda+\beta)}e^{-x}, \quad \gamma-\lambda-\beta>-1, \quad x\in (0, \infty),
\end{equation*}
the constant term has a closed form expression in terms of Gamma functions,
\begin{equation}\label{D47}
(-1)^{n}P_{n}(0;0,\gamma)=\frac{\Gamma(n+1+\gamma-\lambda-\beta)}{\Gamma(1+\gamma-\lambda-\beta)}\sim \sqrt{2\pi}\:
n^{n+\gamma-\alpha-\beta+1/2}\:{\rm e}^{-n},\;\;\;n\to\infty.
\end{equation}

\begin{cor}
With the double scaling, 
 one finds,
\begin{align}\label{D48}
\lim_{n\rightarrow\infty}\frac{(-1)^{n}P_{n}(0;t,\gamma)}{(-1)^{n}P_{n}(0;0,\gamma)}=\frac{\widehat{\Delta}(s,\gamma+1)}{\widehat{\Delta}
(s,\gamma)}=&\exp\left(c_{3}(\gamma)+\frac{\beta+\lambda}{2}\ln{s}+\frac{(1+2\gamma)(\beta+\lambda)}{4}s^{-\frac{1}{2}}\right.\nonumber\\
&\left.-\frac{(1+2\gamma)(\beta+\lambda)^{2}}{16}s^{-1}+\mathcal{O}(s^{-\frac{3}{2}})\right)
\end{align}
where $c_{3}(\gamma)$ is a constant independent of $s.$ Moreover,
\begin{equation}\label{D49}
c_{3}(\gamma)=c_{2}(\gamma+1)-c_{2}(\gamma),
\end{equation}
where $c_{2}(\gamma)$ is the constant in $(\ref{bb00}).$
\end{cor}

\begin{proof}
From the fact that
\begin{align*}
\lim\limits_{n \rightarrow \infty}\frac{(-1)^{n}P_{n}(0; \frac{s}{n},\gamma)}{(-1)^{n}P_{n}(0; 0,\gamma)}=\lim\limits_{n \rightarrow \infty}\frac{\widehat{D}_{n}(s/n, \gamma+1)}{\widehat{D}_{n}(s/n, \gamma)}\frac{\widehat{D}_{n}(0,\gamma)}{\widehat{D}_{n}(0, \gamma+1)}=\frac{\widehat{\Delta}(s,\gamma+1)}{\widehat{\Delta}(s,\gamma)},
\end{align*}
(\ref{D48}) and (\ref{D49}) follow.
\end{proof}
\subsection{Large $n$ behavior of $P_n(0;t,\gamma)$.}
In this subsection, we obtain the constant, $c_{3}(\gamma),$ 
following similar steps in section 2.4. We substitute ${\rm v}(x,t)$ and ${\rm v}'(x,t)$ into $(\ref{a2a})$ and $(\ref{a3}).$
With the aid of the integral formulas in the Appendix A, we find $a$ and $b$ satisfy two algebraic equations,
\begin{equation}\label{D37}
\frac{\lambda+\beta}{\sqrt{(t+a)(t+b)}}-\frac{\gamma}{\sqrt{ab}}+1=0,
\end{equation}
\begin{equation}\label{D38}
\lambda+\beta-\gamma+\frac{a+b}{2}-\frac{t(\lambda+\beta)}{\sqrt{(t+a)(t+b)}}=2n.
\end{equation}
The parameters $a$ and $b$, are, as before, the end points of the support of the equilibrium density.

Let $\widehat{X}:=\sqrt{ab}$, eliminating $a+b$ from the above equations, we see that $\widehat{X}$ satisfies a quintic equation,
\begin{equation}\label{D39}
\left(\widehat{X}^{2}+2t(2n+\gamma-\lambda-\beta+\frac{t\gamma}{\widehat{X}})-t^{2}\right)\left(\gamma-\widehat{X}\right)^{2}
-\widehat{X}^{2}\left(\lambda+\beta\right)^{2}=0.
\end{equation}
Note the large $n$ of $\widehat{X}$ asymptotic of $\hat{X},$ given by,
\begin{equation}\label{D43}
\widehat{X}\sim \gamma-\frac{\gamma(\beta+\lambda)}{2}(nt)^{-\frac{1}{2}}+\frac{\gamma(\beta+\lambda)^{2}}{4}(nt)^{-1}
+\frac{\gamma(\beta+\lambda)(\gamma^{2}-2(\beta+\lambda)^{2})}{16}(nt)^{-\frac{3}{2}}.
\end{equation}
with $t \in (0,t_{1}],$ $0<t_{1}<\infty,$ $n\rightarrow \infty$ such that $nt$ is fixed.
\par
We now give the large $n$ behavior of $P_{n}(0;t,\gamma),$ and place this in the Theorem below, without displaying the detail steps.
\begin{theorem}
If ${\rm v}(x)=-\ln{w(x;t)}=(\lambda+\beta)\ln(x+t)-\gamma\ln{x}+x+t,$ and
\\
$\lambda<0, \beta>0, \gamma>0, t>0, \gamma-\lambda-\beta>-1, x\in (0, \infty),$  the evaluation of the orthogonal polynomials at the $x=0,$ reads,
\begin{align}\label{D45}
(-1)^nP_{n}(0;t,\gamma) \sim n^{n+\gamma-\alpha-\beta+\frac{1}{2}}&e^{-n}\exp\left(-\left(\gamma+\frac{1}{2}\right)\ln{\gamma}+\gamma+\frac{\lambda+\beta}{2}\ln(nt)\right.
\nonumber\\
&\left.+\frac{(1+2\gamma)(\beta+\lambda)}{4}(nt)^{-\frac{1}{2}}-\frac{(1+2\gamma)(\beta+\lambda)^{2}}{16}(nt)^{-1}\right),
\end{align}
the asymptotic estimation is uniform with respect to $t \in (0,t_{1}],$ $0<t_{1}<\infty,$
\\
$\lambda<0, \beta>0, \gamma>0,\gamma-\lambda-\beta>-1,$ $n\rightarrow \infty$ such that $nt$ is fixed.
\end{theorem}
{\bf Remark 7:}
To obtain the constant $c_{3}(\gamma)$ in (\ref{D48}), we rewrite the asymptotic estimation of $(-1)^{n}P_{n}(0;t,\gamma),$ the equation
(\ref{D45}), as
\begin{align*}
(-1)^{n}P_{n}(0;t,\gamma)\sim &n^{n+\gamma-\alpha-\beta+\frac{1}{2}}e^{-n}\exp\left(-\left(\gamma+\frac{1}{2}\right)\ln{\gamma}+\gamma+\frac{\lambda+\beta}{2}\ln(nt)\right.\nonumber\\
&\left.+\frac{(1+2\gamma)(\beta+\lambda)}{4}(nt)^{-\frac{1}{2}}-\frac{(1+2\gamma)(\beta+\lambda)^{2}}{16}(nt)^{-1}\right)\nonumber\\
\end{align*}

Taking into account the extreme right of (\ref{D47}),
we see that,
$$
\frac{P_{n}(0;t,\gamma)}{P_{n}(0;0,\gamma)}\sim \exp\left(c_{3}+\frac{\lambda+\beta}{2}\ln(nt)+\frac{(1+2\gamma)(\beta+\lambda)}{4}(nt)^{-\frac{1}{2}}-\frac{(1+2\gamma)(\beta+\lambda)^{2}}{16}(nt)^{-1}\right),
$$
in accordance to (\ref{D48}). From these facts $c_{3}$ is identify to be
$$
\ln\left(\frac{\Gamma(1+\gamma-\lambda-\beta)}{\gamma}\right)-\left(\gamma-\frac{1}{2}\right)\ln{\gamma}+\gamma-\frac{1}{2}\ln{(2\pi)}.
$$
Consequently, the equation (\ref{D49}) becomes,
$$
c_{2}(\gamma+1)-c_{2}(\gamma)=c_{3}(\gamma)=\ln\left(\frac{\Gamma(1+\gamma-\lambda-\beta)}{\gamma}\right)
-\left(\left(\gamma-\frac{1}{2}\right)\ln{\gamma}-\gamma+\frac{1}{2}\ln{(2\pi)}\right).
$$
\\
We may identify
\begin{equation*}
c_{2}(\gamma)=\ln\left(\frac{G\left(1+\gamma-\lambda-\beta\right)}{\Gamma(\gamma)G(\gamma)}\right)+b(\gamma),\quad \gamma>0
\end{equation*}
up to another constant $b(\gamma),$ which satisfies,
\begin{equation*}
b(\gamma+1)-b(\gamma)=\int_{0}^{\infty}\left(\frac{1}{2}-\frac{1}{t}+\frac{1}{e^{t}-1}\right)\frac{e^{-\gamma{t}}}{t}dt.
\end{equation*}
Binet's formula, see (\cite{WG1996}, P249), says that,
\begin{equation*}
\ln{\Gamma(z)}=\left(z-\frac{1}{2}\right)\ln(z)-z+\frac{1}{2}\ln{(2\pi)}+\int_{0}^{\infty}\left(\frac{1}{2}-\frac{1}{t}
+\frac{1}{e^{t}-1}\right)\frac{e^{-zt}}{t}dt,
\quad {\rm Re\:} z>0.
\end{equation*}
{\bf Remark 8:}
As $\gamma\to\infty,$
$$
b(\gamma+1)-b(\gamma)=\frac{1}{12\:\gamma}-\frac{1}{360\gamma^3}+{\rm O}(1/\gamma^5).
$$
We also note that,
\bea
b(\gamma)=\kappa-\frac{1}{2}\ln{\gamma}+\frac{1}{720\gamma^{2}}++{\rm O}(1/\gamma^4).
\eea
It may be of interest to determine the constant $\kappa$.

\section{Limiting behavior of the Kernel with the Pollaczek-Jacobi type weight.}
The ladder operator formalism for orthogonal polynomials and the associated compatibility condition can be found, for example, in \cite{YMI1997}.
This is a handy tool regarding orthogonal polynomials on the unit circle \cite{BC2008}, construction of  Jacobi polynomials \cite{CI2005} and has been adapted to obtain the
Painlev\'{e} equations, arising the deformation of classical weight. We refer the Reader to \cite{BC2008, CD2010, CF2006, CP2005, CZ2010, FO2010}.
In particular, this r formalism has
been applied to specific orthogonal polynomial ensemble in \cite{ChenIts12010, FO2010}, and its comparison with the isomonodromy approach\cite{JMU1981}. For the extension of ladder operators to discrete orthogonal polynomials, and allied discrete Painlev\'{e} equations; to  $q$-orthogonal polynomials and the allied $q$-Painlev\'{e} equations, we refer the Reader to \cite{B2010, B2011, BA2010, CI2008, IM2010, INS2004, CG2015}.
\\
In this section, we construct the Lax pair for our problem.
\\
First, we recall the ladder operator relations of $(2.7)$ and $(2.8)$ in \cite{CD2010}, satisfied by the polynomials $P_{n}(x;t,\alpha,\beta)$ orthogonal with respect to the Plooaczek-Jacobi type weight,
\begin{equation}\label{101a}
\left(\frac{d}{dx}+B_{n}(x;t,\alpha,\beta)\right)P_{n}(x;t,\alpha,\beta)=\beta_{n}A_{n}(x;t,\alpha,\beta)P_{n-1}(x;t,\alpha,\beta),
\end{equation}
\begin{equation}\label{102a}
\left(\frac{d}{dx}-B_{n}(x;t,\alpha,\beta)-{\rm v}'(x)\right)P_{n-1}(x;t,\alpha,\beta)=-A_{n-1}(x;t,\alpha,\beta)P_{n}(x;t,\alpha,\beta),
\end{equation}
The ``coefficients",
$A_{n}(x;t,\alpha,\beta)$ and $B_{n}(x;t,\alpha,\beta)$ are given by $(2.14),$ $(2.15)$ in \cite{CD2010}, respectively;
\begin{equation*}
A_{n}(x;t,\alpha,\beta)=\frac{R^{*}_{n}}{x^{2}}+\frac{R_{n}}{x}-\frac{R_{n}}{x-1}, \qquad B_{n}(x;t,\alpha,\beta)=\frac{r^{*}_{n}}{x^{2}}-\frac{n-r_{n}}{x}-\frac{r_{n}}{x-1},
\end{equation*}
and $R^{*}_{n}$, $R_{n}$, $r^{*}_{n}$, $r_{n}$ are defined by $(2.16)$-$(2.19)$ in \cite{CD2010}, all depending on $t,$ $\alpha,$ $\beta.$
\par
\begin{theorem}
The monic polynomials $P_{n}(x; t, \alpha, \beta)$ orthogonal with respect to the Pollaczek-Jacobi type weight, satisfy the ladder operator relations in $t$,
\begin{equation}\label{103a}
\left(xt\frac{d}{dt}-r^{*}_{n}\right)P_{n}(x; t, \alpha, \beta)=\beta_{n}R^{*}_{n}P_{n-1}(x; t, \alpha, \beta),
\end{equation}
\begin{equation}\label{104a}
\left(xt\frac{d}{dt}+r^{*}_{n}-t+xR^{*}_{n-1}\right)P_{n-1}(x; t, \alpha, \beta)=R^{*}_{n-1}P_{n}(x; t, \alpha, \beta).
\end{equation}
\end{theorem}

By the Christoffel--Darboux formula \cite{Szego1939}, the reproducing kernel reads,
\begin{equation}\label{a147}
\mathbb{K}_{n}(x,y):=\sqrt{\frac{h_{n}}{h_{n-1}}}\frac{\phi_{n}(x)\phi_{n-1}(y)-\phi_{n}(y)\phi_{n-1}(x)}{x-y},
\end{equation}
where $h_{n}$ is the square of the $L^{2}$ norm is given by (\ref{452}) 
and $\beta_{n}=h_{n}/h_{n-1}.$ Here, $\phi_{n}(x)$ and $\phi_{n-1}(x)$ are defined in terms of $P_n(x;t,\alpha,\beta)$ by
\begin{equation*}
\phi_{n}(x):=\frac{P_{n}(x; t, \alpha, \beta)}{\sqrt{h_{n}}}w^{\frac{1}{2}}(x; t, \alpha, \beta)=\frac{P_{n}(x; t, \alpha, \beta)}{\sqrt{h_{n}}}x^{\frac{\alpha}{2}}(1-x)^{\frac{\beta}{2}}e^{-\frac{t}{2x}},
\end{equation*}
\begin{equation*}
\phi_{n-1}(x):=\frac{P_{n-1}(x; t, \alpha, \beta)}{\sqrt{h_{n-1}}}w^{\frac{1}{2}}(x; t, \alpha, \beta)=\frac{P_{n-1}(x; t, \alpha, \beta)}{\sqrt{h_{n-1}}}x^{\frac{\alpha}{2}}(1-x)^{\frac{\beta}{2}}e^{-\frac{t}{2x}}.
\end{equation*}
Let $\Psi$ ,
\begin{equation*}
\Psi=\begin{bmatrix}
\phi_{n}(x)\\
\phi_{n-1}(x)
\end{bmatrix}
=\begin{bmatrix}
\frac{P_{n}(x; t, \alpha, \beta)}{\sqrt{h_{n}}}x^{\frac{\alpha}{2}}(1-x)^{\frac{\beta}{2}}e^{-\frac{t}{2x}}\\
\frac{P_{n-1}(x; t, \alpha, \beta)}{\sqrt{h_{n-1}}}x^{\frac{\alpha}{2}}(1-x)^{\frac{\beta}{2}}e^{-\frac{t}{2x}}.
\end{bmatrix}
\end{equation*}
We find,
\begin{equation}\label{a201}
\Psi_{x}=\left(\frac{C_{1}}{x-1}+\frac{C_{2}}{x}+\frac{C_{3}}{x^{2}}\right)\Psi,
\end{equation}
and
\begin{equation}\label{105a}
C_{1}=\begin{bmatrix}
r_{n}+\frac{\beta}{2} & -R_{n}\sqrt{\beta_{n}}\\
R_{n-1}\sqrt{\beta_{n}} & -r_{n}-\frac{\beta}{2}
\end{bmatrix}
,
\quad
C_{2}=\begin{bmatrix}
n-r_{n}+\frac{\alpha}{2} & R_{n}\sqrt{\beta_{n}}\\
-R_{n-1}\sqrt{\beta_{n}} & -n+r_{n}-\frac{\alpha}{2}
\end{bmatrix}
,
\end{equation}
\begin{equation}\label{106a}
C_{3}=\begin{bmatrix}
\frac{t}{2}-r_{n}^{*} & R^{*}_{n}\sqrt{\beta_{n}}\\
-R^{*}_{n-1}\sqrt{\beta_{n}} & -\frac{t}{2}+r^{*}_{n}
\end{bmatrix}
.
\end{equation}
Similarly, the ladder operator relationships in $t,$ see (\ref{103a}) and (\ref{104a}) can be re-written as
\begin{equation}\label{a202}
\Psi_{t}=\left(D_{0}+\frac{D_{1}}{x}\right)\Psi,
\end{equation}
 and
\begin{equation}\label{107a}
D_{0}=\begin{bmatrix}
\frac{R_{n}^{*}}{2t} & 0\\
0 & -\frac{R_{n-1}^{*}}{2t}
\end{bmatrix}
,
\qquad
D_{1}=\begin{bmatrix}
\frac{r_{n}^{*}}{t}-\frac{1}{2} & -\frac{R_{n}^{*}\sqrt{\beta_{n}}}{t}\\
\frac{R_{n-1}^{*}\sqrt{\beta_{n}}}{t} & \frac{1}{2}-\frac{r_{n}^{*}}{t}
\end{bmatrix}
.
\end{equation}
The equations (\ref{a201}) and (\ref{a202}) give the Lax pair mentioned earlier. 
The compatibility condition of the Lax pair reads $\Psi_{xt}=\Psi_{tx}$, which, when un-pact, re-produces the differential differences obtained in \cite{CD2010}.
\\
Writing out (\ref{a201}) in component form, followed by eliminating $\phi_{n-1}(x)$ in favor of $\phi_n(x),$ we find,
%
\begin{align}\label{cc03}
&2\left(x^{3}-x^{2}\right)\phi_{n}''(x)+2\left(3x^{2}-2x\right)\phi_{n}'(x)-\left(2(2n+\alpha+\beta)x+2(r_{n}-r^{*}_{n})+t-2n-\alpha\right)\phi_{n}\nonumber\\
&-\frac{2(R^{*}_{n}-R_{n})x^{2}(x-1)}{x(R^{*}_{n}-R_{n})-R^{*}_{n}}\phi_{n}'(x)+\frac{R^{*}_{n}-R_{n}}{x(R^{*}_{n}-R_{n})-R^{*}_{n}}\left((2n+\alpha+\beta)x^{2}
+\left(2(r_{n}-r^{*}_{n})\right.\right.\nonumber\\
&\left.\left.+t-2n-\alpha\right)x+2r^{*}_{n}-t\right)\phi_{n}-\frac{1}{2x^{2}(x-1)}\left[\left(x(2(r_{n}-r^{*}_{n})+t-2n-\alpha)+2r^{*}_{n}-t\right.\right.\nonumber\\
&\left.\left.+x^{2}(2n+\alpha+\beta)\right)^{2}-4\left(x(R^{*}_{n}-R_{n})-R^{*}_{n}\right)\left(x(R^{*}_{n-1}-R_{n-1})-R^{*}_{n-1}\right)\beta_{n}\right]\phi_{n}=0,
\end{align}
where $r_{n}$, $r^{*}_{n}$, $R^{*}_{n}$, $R_{n}$, $R^{*}_{n-1}$, $R_{n-1}$ may be expressed by $H_{n}$, $H_{n}'(t)$ and $H_{n}"(t),$ these are $(2.56),$ $(2.32),$ $(2.33),$ $(3.7),$ $(5.1),$ $(5.2),$ $(5.8)$ in \cite{CD2010}. For convenience, we list these relations in Appendix C.
\\
We give in the next theorem, the scaled kernel.
\begin{theorem}
Let
\begin{equation}\label{441a}
\zeta:=4n^{2}x, \quad \zeta^{*}:=4n^{2}y, \quad s:=2n^{2}t, \quad  \phi(\zeta):=\lim_{n \rightarrow \infty}\phi_{n}\left(\frac{\zeta}{4n^{2}}\right),
\end{equation}
where $n\rightarrow\infty,$ $t\rightarrow0^{+},$  $\zeta$ and $\zeta^{*}$ in compact subsets of $(0,\infty),$ and $s$ finite,
then,
\begin{align}\label{401a0}
&\lim_{n \rightarrow \infty}\frac{1}{2n}\mathbb{K}_{n}\left(\frac{\zeta}{4n^{2}},\frac{\zeta^{*}}{4n^{2}}\right)\nonumber\\
&=\frac{(\zeta-4s{\cal H}')\phi(\zeta){\zeta^{*}}^{2}\phi'(\zeta^{*})-(\zeta^{*}-4s{\cal H}')\phi(\zeta^{*})\zeta^{2}\phi'(\zeta)-2s^{2}{\cal H}''(\zeta-\zeta^{*})\phi(\zeta)\phi(\zeta^{*})}{(\zeta-\zeta^{*})(\zeta-4s{\cal H}')(\zeta^{*}-4s{\cal H}')},
\end{align}
Furthermore, $\phi(\zeta)$ satisfies the ODE,
\begin{align}\label{4a0a138}
\zeta^{2}\phi''(\zeta)+\left(\zeta-\frac{4s{\cal H}'(s)\zeta}{\zeta-4s{\cal H}'(s)}\right)\phi'(\zeta)+\left(\frac{\zeta}{4}+\frac{2s^{2}{\cal H}''(s)}{\zeta-4s{\cal H}'(s)}-\frac{\alpha^{2}}{4}-\frac{s^{2}}{\zeta^{2}}-\frac{\alpha{s}}{\zeta}-{\cal H}\right)\phi(\zeta)=0,
\end{align}
with the boundary conditions of $\phi(0,\alpha,\beta)=0,$  \; $\alpha>0$.
\end{theorem}

\begin{proof}
$\phi_{n-1}(x)$ may be expressed in terms of $\phi_{n}'(x)$ and $\phi_{n}(x)$ via one of the ladder operator relations,
$$
\phi_{n-1}(x)=\frac{\left((2n+\alpha+\beta)x^{2}+(2(n+r_{n}-r_{n}^{*})+t-\alpha)x+2r_{n}^{*}-t\right)\phi_{n}-2(x-1)x^{2}\phi_{n}'(x)}{2\sqrt{\beta_{n}}(xR_{n}
+(1-x)R_{n}^{*})}.
$$
Substituting  $\phi_{n-1},$ given above, $(\ref{441a})$ and the large $n$ asymptotic of $r_{n}^{*}$, $r_{n}$, $R_{n}$ and $R_{n}^{*}$ presented in the Appendix C, into the kernel
given by the equation (\ref{a147}), and let $n\rightarrow\infty,$ then the limiting kernel (\ref{401a0}) is obtained.
\par
We first substitute $\zeta=4n^{2}x$ and $s=2n^{2}t$ into (\ref{cc03}), satisfied by $\phi_{n}(x).$ Recall $r_{n}$, $r^{*}_{n}$, $R^{*}_{n}$, $R_{n}$, $R^{*}_{n-1},$ $R_{n-1}$ in $\cite{CD2010},$ listed in the Appendix C. Now let $n\rightarrow\infty$, then $\phi(\zeta)$ satisfies the second order ODE, namely, (\ref{4a0a138}).
\end{proof}
We list here some properties of the limiting kernel.\\
Let $s=0.$ The boundary condition, ${\cal H}(0,\alpha,\beta)=0$, implies the kernel reduces to
\begin{equation*}
\lim_{n \rightarrow \infty}\frac{1}{2n}\mathbb{K}_{n}\left(\frac{\zeta}{4n^{2}},\frac{\zeta^{*}}{4n^{2}}\right)=\frac{\phi(\zeta)\zeta^{*}\phi'(\zeta^{*})-\phi(\zeta^{*})\zeta\phi'(\zeta)}{\zeta-\zeta^{*}}.
\end{equation*}
In which case, the equation (\ref{4a0a138}) becomes the Bessel differential equation,
with the boundary condition of $\phi(0)=0$. Then the limiting kernel becomes the celebrated Bessel Kernel of Tracy and Widom \cite{TW1994BK}.
\par
For $s>0,$ we make the transformation,
\begin{equation*}
\zeta \rightarrow 4s{\cal H}'(s)\zeta, \quad \zeta^{*} \rightarrow 4s{\cal H}'(s)\zeta^{*}, \quad \phi(\zeta)\rightarrow \zeta^{\rho(s)}\phi(\zeta),
\end{equation*}
 then the kernel (\ref{401a0}) becomes as,
\begin{equation}\label{1bb1}
\lim_{n \rightarrow \infty}\frac{1}{2n}\mathbb{K}_{n}\left(\frac{\zeta}{4n^{2}},\frac{\zeta^{*}}{4n^{2}}\right)=\frac{A(\zeta)B(\zeta^{*})-A(\zeta^{*})B(\zeta)}{\zeta-\zeta^{*}},
\end{equation}
where the functions $A(\zeta)$ and $B(\zeta)$ are given by
\begin{equation*}
A(\zeta):=\frac{\zeta^{\rho(s)}}{2\sqrt{s{\cal H}'(s)}}\:\phi(\zeta),\quad B(\zeta):=\frac{\zeta^{\rho(s)}}{2\sqrt{s{\cal H}'(s)}}\:\frac{\zeta^2}{\zeta-1}\:\phi'(\zeta),
\end{equation*}
and
$$
\rho(s):=\frac{s{\cal H}"(s)}{2{\cal H}'(s)}.
$$
The ODE (\ref{4a0a138}) transforms to,
\begin{equation*}
\phi"(\zeta)+\left(\frac{2(1+b_{0})}{\zeta}-\frac{1}{\zeta-1}\right)\phi'(\zeta)+\left(-\frac{b_{4}}{\zeta^{4}}-\frac{b_{3}}{\zeta^{3}}
+\frac{b_{2}}{\zeta^{2}}+\frac{b_{1}}{\zeta}\right)\phi(\zeta)=0,
\end{equation*}
where
\begin{equation*}
b_{4}=\frac{1}{16}({\cal H}')^{2}, \quad b_{3}=\frac{\alpha}{4{\cal H}'}, \quad b_{2}=\frac{(s{\cal H}'')^{2}-4{\cal H}{{{\cal H}'}^{2}-\alpha^{2}{{\cal H}'}^{2}}}{4{{\cal H}'}^{2}}, \quad b_{1}=s{\cal H}', \quad b_{0}=\frac{s{\cal H}''}{2{\cal H}'}.
\end{equation*}
The coefficients $b_{1}$, $b_{2}$, $b_{3}$, $b_{4}$ satisfy the relation,
$$
b_{1}+b_{2}-b_{3}-b_{4}=0,
$$
which is the $\sigma$-form of the Painlev\'{e} equation found in the Theorem 3.
\begin{equation*}
\left(s{\cal H}''\right)^{2}+4\left({\cal H}'\right)^{2}\left(s{\cal H}'-{\cal H}\right)-\left(\alpha{{\cal H}'}+\frac{1}{2}\right)^{2}=0.
\end{equation*}
\par
The kernel (\ref{1bb1}) has continuity property, a feature of the integrable kernels, since it can be written as,
\begin{equation*}
\lim_{n \rightarrow \infty}\frac{1}{2n}\mathbb{K}_{n}\left(\frac{\zeta}{4n^{2}},\frac{\zeta^{*}}{4n^{2}}\right)=\frac{\sum_{j=0}^{1}f_{j}(\zeta)g_{j}(\zeta^{*})}{\zeta-\zeta^{*}},
\end{equation*}
where
\begin{equation*}
f_{j}(\zeta)=(-1)^{2-j}\frac{\zeta^{\frac{s{\cal H}''}{2{\cal H}'}}}{2\sqrt{s{\cal H}'}}\left(\frac{\zeta^{2}}{\zeta-1}\frac{d}{d\zeta}\right)^{j}\phi(\zeta), \quad g_{j}(\zeta^{*})=\frac{{\zeta}^{\frac{s{\cal H}''}{2{\cal H}'}}}{2\sqrt{s{\cal H}'}}\left(\frac{{\zeta}^{2}}{\zeta-1}\frac{d}{d\zeta}\right)^{1-j}\phi(\zeta),
\end{equation*}
and $j\in\{0,1\},$ with the property $\sum_{j=0}^{1}f_{j}(\zeta)g_{j}(\zeta)=0$.
\par
Using the continuity property, the limiting kernel (\ref{401a0}) becomes,
\begin{align*}
&\lim_{n \rightarrow \infty}\frac{1}{2n}\mathbb{K}_{n}\left(\frac{\zeta}{4n^{2}},\frac{\zeta}{4n^{2}}\right)\nonumber\\
&=\frac{\zeta^{2}\phi'(\zeta)-\zeta^{2}\phi(\zeta)\phi''(\zeta)-2\zeta\phi(\zeta)\phi'(\zeta)}{\zeta-4s{\cal H}'(s)}+\frac{\zeta^{2}\phi(\zeta)\phi'(\zeta)-2s^{2}{\cal H}''(s)\phi^{2}(\zeta)}{(\zeta-4s{\cal H}'(s))^{2}},\\
&=\frac{\zeta^{2}\left(\phi'(\zeta)\right)^{2}+\left(\frac{\zeta}{4}-\frac{\alpha^{4}}{4}-\frac{s^{2}}{\zeta^{2}}-\frac{\alpha{s}}{\zeta}-{\cal H}\right)\phi^{2}(\zeta)}{\zeta-4s{\cal H}'}, \quad s>0,\\
&=\frac{\zeta^{2}\left(\phi'(\zeta)\right)^{2}+\left(\frac{\zeta}{4}-\frac{\alpha^{2}}{4}\right)\phi^{2}(\zeta)}{\zeta},\quad s=0,\\
&=\lambda\frac{J_{\alpha}^{2}(\sqrt{\zeta})-J_{\alpha-1}\left(\sqrt{\zeta}\right)J_{\alpha+1}\left(\sqrt{\zeta}\right)}{4},
\end{align*}
where $J_{\alpha}(z)$ is the Bessel function of order $\alpha$, $\lambda$ is a parameter. The first equality above, is found by  applying L'Hospital rule to $(\ref{401a0})$, the second equality is achieved by eliminating $\phi''(\zeta)$ with the aid of $(\ref{4a0a138})$ and ${\cal H}(0,\alpha,\beta)=0$. The Bessel equation satisfied by $\phi(\zeta)$ with the boundary condition $\phi(0)=0$ has regular solution $\phi(\zeta)=\lambda\:J_{\alpha}(\sqrt{\zeta})$ with a parameter of ${\lambda}.$

\section{Limiting Kernel with the perturbed Laguerre weight.}
In this section, we make use of the method in the previous section, and apply to the kernel generated by the singularly perturbed Laguerre weight. It is  interesting  that the
 limiting kernel via from the perturbed Laguerre weight is the same with the limiting kernel that arises from the Pollaczek Jacobi type weight, although their scaling schemes are quite different.
\par
First, the monic orthogonal polynomials $P_{n}(x;t,\alpha)$ with respect to the singular perturbed Laguerre weight,
\begin{equation*}
w(x;t,\alpha)=x^{\alpha}e^{-x}e^{-\frac{t}{x}}, \quad 0 \leq x <\infty, \quad \alpha>0, \quad t>0,
\end{equation*}
satisfy the ladder operator relations have been derived in $\cite{ChenIts12010}$, which we restate here,
\begin{equation*}
\left(\frac{d}{dx}+B_{n}(x)\right)P_{n}(x;t,\alpha)=\beta_{n}A_{n}(x)P_{n-1}(x;t,\alpha),
\end{equation*}
\begin{equation*}
\left(\frac{d}{dx}-B_{n}(x)-{\rm v'}(x)\right)P_{n-1}(x;t,\alpha)=-A_{n-1}(x)P_{n}(x;t,\alpha),
\end{equation*}
where ${\rm v}(x)=-\ln{w(x;t,\alpha)}$. The coefficients  $A_{n}(x),$ $B_{n}(x)$ are given by $(2.7)$ and $(2.8)$ in $\cite{ChenIts12010}$,
\begin{equation*}
A_{n}(x)=\frac{1}{x}+\frac{R_{n}}{x^{2}}, \quad B_{n}(x)=-\frac{n}{x}+\frac{r_{n}}{x^{2}}.
\end{equation*}
We use $R_{n}(t),$ $r_{n}(t)$ and $t$ instead of $a_{n}(s)$, $b_{n}(s)$ and $s$ respectively in $\cite{ChenIts12010}$; these are defined by,
\begin{equation*}
R_{n}(t):=\frac{t}{h_{n}}\int_{0}^{\infty}\frac{1}{y}P_{n}^{2}(y;t,\alpha)w(y;t,\alpha)dy, \quad r_{n}(t):=\frac{t}{h_{n-1}}\int_{0}^{\infty}\frac{1}{y}P_{n}(y;t,\alpha)P_{n-1}(y;t,\alpha)w(y)dy,
\end{equation*}
depend on $t$ and $\alpha$, with the initial conditions $R_{n}(0)=0,$ $r_{n}(0)=0$.
\par
The monic orthogonal polynomials $P_{n}(x;t,\alpha)$ satisfy ladder operator relations in $t$, see $(5.56)$ and $(5.57)$ in  $\cite{ChenIts12010},$
\begin{equation*}
\left(\frac{d}{dt}-\frac{r_{n}}{xt}\right)P_{n}(x; t, \alpha)=-\frac{\beta_{n}R_{n}}{xt}P_{n-1}(x; t, \alpha),
\end{equation*}
\begin{equation*}
\left(\frac{d}{dt}-\frac{1}{x}+\frac{r_{n}}{xt}+\frac{R_{n-1}}{t}\right)P_{n-1}(x; t, \alpha )=\frac{R_{n-1}}{xt}P_{n}(x; t, \alpha).
\end{equation*}
\par
The reproducing Kernel with respect to $w(x;t,\alpha)$ is given by,
\begin{equation}\label{a138}
\mathbf{K}_{n}(x,y):=\left(\frac{h_{n}}{h_{n-1}}\right)^{\frac{1}{2}}\frac{\varphi_{n}(x)\varphi_{n-1}(y)-\varphi_{n}(y)\varphi_{n-1}(x)}{x-y},
\end{equation}
where,
\begin{align}
&\varphi_{n}(x):=\frac{P_{n}(x;t,\alpha)}{\sqrt{h_{n}}}w^{\frac{1}{2}}(x)=\frac{P_{n}(x;t,\alpha)}{\sqrt{h_{n}}}x^{\frac{\alpha}{2}}e^{-\frac{x}{2}-\frac{t}{2x}},\nonumber\\
&\varphi_{n-1}(x):=\frac{P_{n-1}(x;t,\alpha)}{\sqrt{h_{n-1}}}w^{\frac{1}{2}}(x)=\frac{P_{n-1}(x;t,\alpha)}{\sqrt{h_{n-1}}}x^{\frac{\alpha}{2}}e^{-\frac{x}{2}-\frac{t}{2x}},\nonumber
\end{align}
where $h_{n}$ is the square of the $L^{2}$ norm 
and  $\beta_{n}=h_{n}/h_{n-1}$.

Let $\symbol{'010}(x)$
\begin{equation*}
\symbol{'010}:=\begin{bmatrix}
\varphi_{n}(x)\\
\varphi_{n-1}(x)
\end{bmatrix}
=\begin{bmatrix}
\frac{P_{n}(x;t,\alpha)}{\sqrt{h_{n}}}x^{\frac{\alpha}{2}}e^{-\frac{x}{2}-\frac{t}{2x}}\\
\frac{P_{n-1}(x;t,\alpha)}{\sqrt{h_{n-1}}}x^{\frac{\alpha}{2}}e^{-\frac{x}{2}-\frac{t}{2x}}
\end{bmatrix}.
\end{equation*}
We find,
\begin{equation}\label{a142}
\symbol{'010}_{x}=\left(E_{0}+\frac{E_{1}}{x}+\frac{E_{2}}{x^{2}}\right)\symbol{'010},
\end{equation}
where $\symbol{'010}_{x}$ denotes $\frac{\partial}{\partial{x}}\symbol{'010}$,
\begin{equation*}
E_{0}=\frac{1}{2}\begin{bmatrix}
-1 & 0\\
0  & 1
\end{bmatrix}
,
\qquad
E_{1}=\begin{bmatrix}
n+\frac{\alpha}{2} &  \sqrt{\beta_{n}}\\
-\sqrt{\beta_{n}} & -n-\frac{\alpha}{2}
\end{bmatrix}
,
\qquad
E_{2}=\begin{bmatrix}
\frac{t}{2}-r_{n} & R_{n}\sqrt{\beta_{n}}\\
-R_{n-1}\sqrt{\beta_{n}} & -\frac{t}{2}+r_{n}
\end{bmatrix}.
\end{equation*}
The quantity $\symbol{'010}$ also satisfies,
\begin{equation}\label{a143}
\symbol{'010}_{t}=\left(F_{0}+\frac{F_{1}}{x}\right)\symbol{'010},
\end{equation}
where
\begin{equation*}
F_{0}=\begin{bmatrix}
\frac{R_{n}}{2t} & 0\\
0 & -\frac{R_{n-1}}{2t}
\end{bmatrix}
,
\qquad
F_{1}=\begin{bmatrix}
\frac{r_{n}}{t}-\frac{1}{2} & -\frac{R_{n}\sqrt{\beta_{n}}}{t}\\
\frac{R_{n-1}\sqrt{\beta_{n}}}{t} & \frac{1}{2}-\frac{r_{n}}{t}
\end{bmatrix}¡£
\end{equation*}
The equations (\ref{a142}) and (\ref{a143}) is the Lax pair of this problem.
Un-packing the compatibility condition is satisfied by the Lax pair, 
we find the following set of scalar equations,
\begin{equation}\label{a144}
t\frac{d}{dt}r_{n}=\frac{2(r_{n}^{2}-tr_{n})}{R_{n}}+(2n+1+\alpha)r_{n}-nt, \nonumber
\end{equation}
\begin{equation}\label{a145}
t\frac{d}{dt}R_{n}=R_{n}^{2}+(2n+1+\alpha)R_{n}-t+2r_{n},
\end{equation}
\begin{equation}\label{a146}
t\frac{d}{dt}R_{n-1}=t-2r_{n}-(2n-1+\alpha+R_{n-1})R_{n-1}.
\end{equation}
Note that the above two Riccati equations are the same with $(3.10)$ and $(3.11)$ in \cite{ChenIts12010}, the third equation in the above also can be derived by $(2.16)$ and $(3.10)$ in \cite{ChenIts12010}.
\par
We write $(\ref{a142})$ as a set of scalar equations,
\begin{equation}\label{a136}
\varphi_{n}'(x)=\left(-\frac{1}{2}+\frac{2n+\alpha}{2x}+\frac{t-2r_{n}}{2x^{2}}\right)\varphi_{n}(x)+\left(\frac{1}{x}+\frac{R_{n}}{x^{2}}\right)
\beta_{n}^{\frac{1}{2}}\varphi_{n-1}(x),
\end{equation}
\begin{equation}\label{a137}
\varphi_{n-1}'(x)=-\left(\frac{1}{x}+\frac{R_{n-1}}{x^{2}}\right)\beta_{n}^{\frac{1}{2}}\varphi_{n}(x)-\left(-\frac{1}{2}+\frac{2n+\alpha}{2x}+\frac{t-2r_{n}}{2x^{2}}\right)
\varphi_{n-1}(x).
\end{equation}
If we eliminate $\varphi_{n-1}(x)$ 
The resulting ODE reads,
\begin{align}\label{a140}
&x^{2}\varphi_{n}''(x)+\left(x+\frac{xR_{n}}{x+R_{n}}\right)\varphi_{n}'(x)-\left(\frac{x-R_{n}}{2}+\frac{t\frac{d}{dt}R_{n}-R_{n}}{2(x+R_{n})}\right)\varphi_{n}(x)\nonumber\\
&-\left(\frac{x^{2}}{4}+\frac{\alpha^{2}}{4}+\frac{t^{2}}{4x^{2}}-x\left(n+1+\frac{\alpha}{2}\right)-\frac{t}{2}+\frac{\alpha{t}}{2x}+H_{n}\right)\varphi_{n}(x)=0,
\end{align}
with the aid of $(\ref{a146}),$ and $(2.12),$ $(2.13),$ $(3.21),$ $(3.29)$ in \cite{ChenIts12010}.\\
In order to investigate the limiting behavior of the kernel, we recall the equations of $(2.12),$ $(2.13),$ $(3.21),$ $(3.29)$ in \cite{ChenIts12010}
\begin{equation*}
\quad \beta_{n}(R_{n}+R_{n-1})=-(2n+\alpha)r_{n}+nt,
\end{equation*}
\begin{equation*}
 \quad \beta_{n}R_{n}R_{n-1}=r_{n}^{2}-tr_{n}.
\end{equation*}
\begin{equation*}
\quad r_{n}=tH_{n}', \qquad  \quad \beta_{n}=n^{2}+n\alpha+tH_{n}'-H_{n}.
\end{equation*}
The equations,$(3.28)$ in \cite{ChenIts12010} becomes,
$$
\alpha_{n}=2n+1+\alpha+\frac{2\left(t(H_{n}')^{2}-tH_{n}'\right)}{tH_{n}''+n-(2n+\alpha)H_{n}'},
$$
and with $(2.9)$ in \cite{ChenIts12010}, we find
\begin{equation}\label{01b0}
R_{n}=\frac{2\left(t(H_{n}')^{2}-tH_{n}'\right)}{tH_{n}''+n-(2n+\alpha)H_{n}'},
\end{equation}
Let
\begin{equation*}
{\cal H}(s,\alpha):=\lim_{n \rightarrow \infty}H_{n}\left(\frac{s}{2n+1+\alpha},\alpha\right).
\end{equation*}
\\
In the next theorem we describe  the scaled kernel.
\begin{theorem}
Let
\begin{equation}\label{41a}
\xi:=4nx, \quad \xi^{*}:=4ny, \quad s:=(2n+1+\alpha)t, \quad {\rm and} \quad \varphi(\xi):=\lim_{n \rightarrow \infty}\varphi_{n}\left(\frac{\xi}{4n}\right),
\end{equation}
where  $t\rightarrow0^{+},$ $n\rightarrow\infty$ and $\xi,$ $\xi^{*}$ are in compact subsets of $(0,\infty),$ $s$ finite, $\alpha>0,$ then
\begin{align}\label{01a0}
&\lim_{n \rightarrow \infty}\frac{1}{4n}\mathbf{K}_{n}\left(\frac{\xi}{4n},\frac{\xi^{*}}{4n}\right)\nonumber\\
&=\frac{(\xi-4s{\cal H}'(s))\varphi(\xi){\xi^{*}}^{2}\varphi'(\xi^{*})-(\xi^{*}-4s{\cal H}'(s))\varphi(\xi^{*})\xi^{2}\varphi'(\xi)-2s^{2}{\cal H}''(s)(\xi-\xi^{*})\varphi(\xi)\varphi(\xi^{*})}{(\xi-\xi^{*})(\xi-4s{\cal H}'(s))(\xi^{*}-4s{\cal H}'(s))},
\end{align}
and $\varphi(\xi)$ satisfies the second order ODE,
\begin{align}\label{a0a138}
\xi^{2}\varphi''(\xi)+\left(\xi-\frac{4s{\cal H}'(s)\xi}{\xi-4s{\cal H}'(s)}\right)\varphi'(\xi)+\left(\frac{\xi}{4}+\frac{2s^{2}{\cal H}''(s)}
{\xi-4s{\cal H}'(s)}-\frac{\alpha^{2}}{4}-\frac{s^{2}}{\xi^{2}}-\frac{\alpha{s}}{\xi}-{\cal H}(s)\right)\varphi(\xi)=0,
\end{align}
with the boundary conditions $\varphi(0)=0,$
\end{theorem}
\begin{proof}
By (\ref{a136}), $\varphi_{n-1}(x)$ can be represented by $\varphi_{n}(x)$ and its derivative as,
$$
\varphi_{n-1}(x)=\frac{1}{(x+R_{n})\sqrt{\beta_{n}}}\left(x^{2}\varphi_{n}'x)-\left(-\frac{1}{2}x^{2}+\left(n+\frac{\alpha}{2}\right)x+\frac{t}{2}-r_{n}\right)
\varphi_{n}(x)\right).
$$
Inserting the above equation and (\ref{41a}) into the kernel given by (\ref{a138}) 
, let $n\rightarrow\infty$, then we arrive at (\ref{01a0}) and (\ref{a0a138}).
\par
In deriving these results, a number asymptotic relations, are listed here. 
$$
r_{n}=tH_{n}'(t)\sim s{\cal H}'(s), \quad R_{n} \sim -\frac{s{\cal H}'(s)}{n}, \quad t\frac{d}{dt}R_{n} \sim -\frac{s{\cal H}'(s)}{n}-\frac{s^{2}{\cal H}''(s)}{n},
$$
where $t\rightarrow0,$ $n \rightarrow \infty$ and $s=(2n+1+\alpha)t$ is fixed.
\end{proof}
 Hence, we see that the scaling limit of the logarithmic derivative of the Hankel determinant $H_{n}(t,\alpha)$ generated by the perturbed Laguerre weight is the same with the scaling limit of logarithmic derivative of the Hankel determinant generated by the Pollaczek-Jacobi type kernel. 
 Moreover, the limiting kernel $(\ref{01a0})$ is completely characterized by the second order linear ODE of $(\ref{a0a138}),$ which is the same with the limiting Pollaczek-Jacobi type kernel in the Theorem 13, both of which are ``scaled at the origin" but their scaling scheme are different from each other.
\\
We present here  another version of the  ODE $(\ref{a0a138}).$\\
For $s>0$, we make the transformation,
\begin{equation*}
\xi \rightarrow 4s{\cal H}'\xi,
\end{equation*}
where ${\cal H}'$ denotes $d{\cal H}(s)/ds$, into $(\ref{a0a138});$ we see that $\varphi(\xi)$ satisfies,
\begin{equation*}
\varphi''(\xi)+\left(\frac{2}{\xi}-\frac{1}{\xi-1}\right)\varphi'(\xi)+\left(-\frac{a_{4}}{\xi^{4}}-\frac{a_{3}}{\xi^{3}}+\frac{a_{2}}{\xi^{2}}
+\frac{a_{1}}{\xi}+\frac{a_{0}}{\xi-1}\right)\varphi(\xi)=0,
\end{equation*}
where,
\begin{equation*}
a_{4}=\frac{1}{16{\cal H}'^{2}}, \quad a_{3}=\frac{\alpha}{4{\cal H}'}, \quad a_{2}=-{\cal H}-\frac{\alpha^{2}}{4}-\frac{s{\cal H}''}{2{\cal H}'}, \quad a_{1}=s{\cal H}'-\frac{s{\cal H}''}{2{\cal H}'}, \quad a_{0}=\frac{s{\cal H}''}{2{\cal H}'},
\end{equation*}
The coefficients of $a_{4}$, $a_{3}$, $a_{2}$, $a_{1}$ and $a_{0}$ satisfies the relation,
\begin{equation*}
a_{0}^{2}+2a_{0}+a_{1}+a_{2}-a_{3}-a_{4}=0,
\end{equation*}
which is the $\sigma$-form of the Painlev\'{e} equation,
\begin{equation*}
\left(s{\cal H''}\right)^{2}+4\left({\cal H'}\right)^{2}\left(s{\cal H'}-{\cal H}\right)-\left(\alpha{{\cal H'}}+\frac{1}{2}\right)^{2}=0.
\end{equation*}
\\
We consider an interesting special case, where $\alpha=0$ The ODE becomes
\begin{equation*}
\varphi''(\xi)+\left(\frac{2}{\xi}-\frac{1}{\xi+2s^{\frac{2}{3}}}\right)\varphi'(\xi)+\left(-\frac{s^{2}}{\xi^{4}}
+\frac{27s^{\frac{2}{3}}+5}{36\xi^{2}}+\frac{3s^{\frac{2}{3}}-1}{12s^{\frac{2}{3}}\xi}+\frac{1}{12({\xi}s^{\frac{2}{3}}+2s^{\frac{4}{3}})}\right)\varphi(\xi)=0,
\end{equation*}
with the boundary condition of $\varphi(0)=0.$ The solution of the $\sigma$-from Painlev\'{e} equation, in this situation, reads,
\begin{equation*}
{\cal H}(s)=-\frac{3}{4}s^{\frac{2}{3}}+\frac{1}{36}.
\end{equation*}


{\bf The region $\xi<<1,$ and $s>>1.$} The approximating equation is
\begin{equation*}
\psi"(\xi)+\frac{2}{\xi}\:\psi'(\xi)-\frac{s^{2}}{\xi^{4}}\psi(\xi)=0,
\end{equation*}
with boundary condition $\psi(0)=0$. The solution, up to a constant multiplier, is given by
\begin{equation*}
e^{-\frac{s}{\xi}}.
\end{equation*}
{\bf The region $\xi>>1, {\xi}s^{-\frac{2}{3}}>>1.$} The approximating equation is
\begin{equation*}
\psi''(\xi)+\frac{1}{\xi}\psi'(\xi)+\left(\frac{3s^{\frac{2}{3}}-1}{12s^{\frac{2}{3}}\xi}
+\frac{1}{12s^{\frac{2}{3}}(\xi+2s^{\frac{2}{3}})}\right)\psi(\xi)=0,
\end{equation*}
we find that, $\varphi(\xi)$ is asymptotic to
$$
A\;(\xi+2s^{\frac{2}{3}}){\rm HeunC}\left(0,1,0,\frac{1}{2}s^{\frac{2}{3}},\frac{1}{3},\frac{\xi}{2}s^{-\frac{2}{3}}+1\right),
$$
where $A$ is an arbitrary constant.
\\

{\bf Acknowledgement.}
\par
We would like to thank the Macau Science and Technology Development Fund for generous support: FDCT 077/2012/A3, and the National Science Foundation of China (ProjectNo.11271079), Doctoral Programs Foundation of the Ministry of Education of China.
\\

{ {\bf Appendix A}}
\par
We list here a selection of integral formulas that are relevant for the computations in the main text. These can also be found in \cite{YM2012} and \cite{YHM2013}.
For $0<a<b,$ we have,
$$
\int_{a}^{b}\frac{1}{\sqrt{(b-x)(x-a)}}dx=\pi. \eqno{(A1)}
$$

$$
\int_{a}^{b}\frac{x}{\sqrt{(b-x)(x-a)}}dx=\frac{(a+b)\pi}{2}. \eqno{(A2)}
$$

$$
\int_{a}^{b}\frac{1}{x\sqrt{(b-x)(x-a)}}dx=\frac{\pi}{\sqrt{ab}}. \eqno{(A3)}
$$

$$
\int_{a}^{b}\frac{1}{x^2\sqrt{(b-x)(x-a)}}dx=\frac{(a+b)\pi}{2(ab)^{\frac{3}{2}}}. \eqno{(A4)}
$$

$$
\int_{a}^{b}\frac{\ln{x}}{\sqrt{(b-x)(x-a)}}dx=2\pi\ln\left(\frac{\sqrt{a}+\sqrt{b}}{2}\right). \eqno{(A5)}
$$

$$
\int_{a}^{b}\frac{\ln{x}}{x\sqrt{(b-x)(x-a)}}dx=\frac{2\pi}{\sqrt{ab}}\ln\frac{2\sqrt{ab}}{\sqrt{a}+\sqrt{b}}. \eqno{(A6)}
$$
\\
\\

{ {\bf Appendix B}}
\par
We list here a number of identities involving $R_n,\:R_n*,\;r_n,\;r_n*,\bt_n,\;H_n$ which can be found in \cite{CD2010},
\begin{equation*}
R_{n}^{*}=R_{n}-(2n+1+\alpha+\beta),
\end{equation*}

\begin{equation*}
(r_{n}^{*})^{2}-tr_{n}^{*}=\beta_{n}R_{n}^{*}R_{n-1}^{*},
\end{equation*}

\begin{equation*}
 r_{n}^{2}+\beta{r_{n}}=\beta_{n}R_{n}R_{n-1},
\end{equation*}

\begin{equation*}
\beta_{n}=\frac{-(r_{n}^{*}-r_{n})^{2}-(\beta+t)r_{n}+(t-\alpha-2n)r_{n}^{*}+nt}{1-(2n+\alpha+\beta)^{2}},
\end{equation*}

\begin{equation*}
r_{n}^{*}=\frac{nt+tH_{n}'}{2n+\alpha+\beta},
\end{equation*}

\begin{equation*}
 r_{n}=\frac{n(n+\alpha)+tH_{n}'-H_{n}}{2n+\alpha+\beta},
\end{equation*}

\begin{equation*}
R_{n}=\frac{(2n+1+\alpha+\beta)[2r_{n}^{2}+(t+2\beta-2r_{n}^{*})r_{n}+(2n+\alpha)r_{n}^{*}-nt-tr_{n}']}{2[(r_{n}^{*}-r_{n})^{2}+(2n+\alpha-t)r_{n}^{*}+(\beta+t)r_{n}-nt]}.
\end{equation*}
The expressions of these for large $n$, are given below,
$$
r_{n}^{*}(t)=\frac{s{\cal H}'}{2n}+\frac{s-(\alpha+\beta)s{\cal H}'}{n^{2}}+\mathcal{O}(\frac{1}{n^3}),
$$

$$
r_{n}=\frac{n}{2}+\frac{\alpha-\beta}{4}+\frac{\beta^{2}-\alpha^{2}-4({\cal H}-s{\cal H})}{8n}+\frac{(\alpha+\beta)((\alpha^{2}-\beta^{2})+4({\cal H}-s{\cal H}))}{16n^{2}}+\mathcal{O}(\frac{1}{n^3}),
$$

$$
R_{n}=2n+1+\alpha+\beta-\frac{2s{\cal H}'}{n}+\frac{(\alpha+\beta-1)s{\cal H}'-2s^{2}{\cal H}''}{n^{2}}+\mathcal{O}(\frac{1}{n^3}),
$$

$$
R_{n}^{*}=-\frac{2s{\cal H}'}{n}+\frac{(\alpha+\beta-1)s{\cal H}'-2s^{2}{\cal H}''}{n^{2}}+\mathcal{O}(\frac{1}{n^3}),
$$
where ${\cal H}$ is given by
$$
{\cal H}(s,\alpha,\beta):=\lim_{n \rightarrow \infty}H_{n}\left(\frac{s}{2n^{2}},\alpha,\beta\right),
$$
${\cal H}'$ denotes $d{\cal H}(s)/ds$ and $s=2n^{2}t$.

\end{document}